\documentclass[a4paper,11pt]{article}

\usepackage[T1]{fontenc}
\usepackage{ae,aecompl}
\usepackage{amsfonts}
\usepackage{amssymb}
\usepackage{amsmath}
\usepackage{amsthm}
\usepackage{verbatim}
\usepackage{hyperref}
\usepackage{xcolor}
\usepackage{xspace}
\usepackage{enumitem}

\usepackage{tikz}
\usetikzlibrary{matrix, positioning, decorations.pathreplacing, shapes.multipart}

\usepackage[top=1in, bottom=1.25in, left=1.25in, right=1.25in]{geometry}

\setitemize{noitemsep,topsep=5pt,parsep=5pt,partopsep=5pt}
\setenumerate{noitemsep,topsep=5pt,parsep=5pt,partopsep=5pt}

\theoremstyle{plain}
\newtheorem{theorem}{Theorem}[section]
\newtheorem{lemma}[theorem]{Lemma}
\newtheorem{corollary}[theorem]{Corollary}
\newtheorem{proposition}[theorem]{Proposition}
\theoremstyle{definition}
\newtheorem{definition}[theorem]{Definition}
\newtheorem{remark}[theorem]{Remark}
\newtheorem{example}[theorem]{Example}

\DeclareMathOperator{\Agt}{\mathbb{A}gt} 
\DeclareMathOperator{\St}{St} 
\DeclareMathOperator{\Prop}{\Pi} 
\DeclareMathOperator{\Act}{Act} 
\DeclareMathOperator{\paths}{\mathsf{paths}} 
\newcommand{\void}{\ensuremath{\mathsf{void}\xspace}}
\newcommand{\card}{\ensuremath{\mathsf{card}\xspace}}
\newcommand{\open}{\ensuremath{\mathsf{open}\xspace}}
\newcommand{\avact}{\ensuremath{\mathsf{action}\xspace}}
\newcommand{\tlimbound}{\Gamma} 

\RequirePackage{ifthen}

\newboolean{reportversion}

\setboolean{reportversion}{false}

\newcommand\refhereortechnicalreport{
\ifthenelse{\boolean{reportversion}}{
\!Section \ref{technicalreport}}
{
\!\!the technical report \cite{technicalreport}}
}

\newcommand{\techrep}[1]{{\color{black} #1}}
\newcommand{\aamasp}[1]{}

\newcommand{\vcut}[1]{}

\newcommand{\rnote}[1]{{\ \color{green}\bf{}{RR: #1}}}

\newcommand{\acta}{\ensuremath{\alpha\xspace}}

\newcommand{\coop}[2][]{\langle\!\langle{#2}\rangle\!\rangle_{_{\!\mathit{#1}}}\,}

\newcommand{\atlx}{\mathord \mathsf{X}\, }
\newcommand{\atlf}{\mathord \mathsf{F}\, }
\newcommand{\atlg}{\mathord \mathsf{G}\, }
\newcommand{\atlu}{\, \mathsf{U} \, }
\newcommand{\atlr}{\, \mathsf{R} \, }
\newcommand{\X}{\atlx}
\newcommand{\F}{\atlf}
\newcommand{\G}{\atlg}
\newcommand{\U}{\atlu}

\newcommand{\bbN}{\ensuremath{\mathbb{N}}}

\newcommand{\atla}{\varphi}
\newcommand{\atlb}{\psi}

\newcommand{\atlsa}{\Phi}

\newcommand{\Logicname}[1]{\ensuremath{\mathsf{#1}}}

\newcommand{\ATL}{\Logicname{ATL}\xspace}
\newcommand{\ATLs}{\Logicname{ATL^*}\xspace}
\newcommand{\ATLplus}{\Logicname{ATL^+}\xspace}
\newcommand{\GTS}{\Logicname{GTS}\xspace}
\newcommand{\CGM}{\Logicname{CGM}\xspace}

\newcommand{\vrf}{\ensuremath{\mathbf{V}\xspace}}
\newcommand{\ovrf}{\ensuremath{\overline{\vrf}\xspace}}

\newcommand{\ctr}{\ensuremath{\mathbf{S}\xspace}}
\newcommand{\octr}{\ensuremath{\overline{\mathbf{S}}\xspace}}
\newcommand{\pl}{\ensuremath{\mathbf{P}\xspace}}
\newcommand{\opl}{\ensuremath{\mathbf{\overline{P}}\xspace}}
\newcommand{\qzero}{\ensuremath{q_{in}\xspace}}
\newcommand{\evalgame}{\ensuremath{\mathcal{G}\xspace}}

\newcommand{\step}{\ensuremath{\mathsf{step}\xspace}}
\newcommand{\bd}{\ensuremath{\mathsf{BD}\xspace}}
\newcommand{\rbb}{\ensuremath{\mathsf{RBB}\xspace}}

\newcommand{\subf}{\ensuremath{\mathrm{SUB}\xspace}}

\newcommand{\defstyle}{\textbf}

\newcommand{\runex}[1]{{\small \textit{ #1}}}

\newcommand{\lgt}{\ensuremath{\mathrm{lgt}\xspace}}

\title{Game-Theoretic Semantics for ATL{\textsuperscript{\large \textbf{+}}} \\ with Applications to Model Checking
}

\date{}

\author{
  Valentin Goranko\footnote{\texttt{valentin.goranko@philosophy.su.se}}\\
  Stockholm University\\
  Sweden
  \and
  Antti Kuusisto\footnote{\texttt{antti.kuusisto@tuni.fi}}\\
  Tampere University\\
  Finland
  \and
  Raine R\"{o}nnholm\footnote{\texttt{raine.ronnholm@tuni.fi}}\\
  Tampere University\\
  Finland
}

\begin{document}

\maketitle

\begin{abstract}
We develop a game-theoretic semantics (\GTS) for the fragment $\ATLplus$ of the alternating-time temporal logic 
$\ATLs$, thereby extending the recently introduced \GTS for \ATL.
We show that the game-theoretic semantics is equivalent to the standard compositional semantics of \ATLplus with perfect-recall strategies. 
Based on the new semantics, we provide an analysis of the memory and time resources needed for model checking $\ATLplus$ and show that strategies of the verifier that use only a very limited amount of memory suffice. Furthermore, using the \GTS, we provide a new algorithm for model checking \ATLplus and identify a natural hierarchy of tractable fragments of \ATLplus that substantially extend \ATL. 

\end{abstract}

%

\section{Introduction}
\label{sect:intro}


The full Alternating-time Temporal Logic \ATLs
\cite{AHK02} is one of the main logical systems used for formalising and verifying strategic reasoning about agents in multi-agent systems. It is very expressive, and that expressiveness comes at a high (2-EXPTIME) price of computational complexity of model checking. Its basic
fragment \ATL ---which can be regarded as the multi-agent extension of CTL--- has, on the other hand, tractable model checking but its expressiveness is rather limited. In particular, \ATL only allows expressing strategic objectives of the type $\coop{A}\Phi$ where $\Phi$ is a simple temporal goal involving a single temporal operator. The intermediate fragment  \ATLplus naturally emerges as a good alternative, essentially extending \ATL to allow  expressing strategic objectives which are Boolean combinations of simple temporal goals. The price for this is a reasonably higher computational complexity of model checking \ATLplus, viz. PSPACE-completeness \cite{BJ}. Still, the PSPACE-completeness result alone gives a rather crude estimate of the amount of computational resources, such as memory, needed for model checking \ATLplus.


\textbf{Main ideas and contributions.} 
In this paper we take an alternative approach to the semantic analysis and model checking of fragments of $\ATLs$, concentrating in particular on fragments of $\ATLplus$. Our analysis is not based on the standard compositional semantics but on a new, game-theoretic semantics (\GTS). The main aims and contributions of the paper are three-fold:

\begin{enumerate}
\item We introduce an adequate game-theoretic semantics for $\ATLplus$ equivalent to the standard (perfect-recall) compositional semantics.

\item We propose new model checking algorithms for \ATLplus and some of its fragments, using the \GTS developed here, rather than the standard semantics. We also analyse more precisely the use of memory resources in \ATLplus via \GTS.

\item We apply the \GTS-based approach to model checking in order to identify new tractable fragments of $\ATLplus$.
\end{enumerate}

The main part of the paper consists of a detailed presentation and analysis of the new \GTS for $\ATLplus$.
In particular, we obtain results similar to those in our earlier work \cite{GKR-AAMAS2016}, where we defined a \GTS for $\ATL$.
We establish, inter alia, the surprising result 
that it is always sufficient to
consider \emph{finite paths only} when 
formulae are evaluated via \GTS, \emph{even when considering infinite models}.
Since we are dealing with $\ATLplus$ as
opposed to $\ATL$, a range of new technical ideas and mechanisms are
needed for the correct evaluation of multiple temporal
goals pursued simultaneously by the proponent coalition.

The approach via \GTS enables us, inter alia, to perform a more precise analysis on the memory resources needed for evaluating \ATLplus-formulae
than the algorithm from \cite{BJ}  
which employs a mix of a path construction procedure for checking strategic formulae $\coop{A}\Phi$ on one hand, and the standard labelling algorithm
on the other hand.
%
%
Our model checking algorithm for \ATLplus follows uniformly a procedure directly based on \GTS and also enables us, inter alia, to identify and correct a flaw in the model checking procedure of \cite{BJ} and some of the claims on which it is based. Yet, the PSPACE upper bound result of \cite{BJ} is easily confirmed by our algorithm, and we provide a new  simple proof of that result. Besides new methods, we use some nice ideas from \cite{BJ}.

As a new complexity result obtained via \GTS, we identify a
natural hierarchy of fragments of \ATLplus that 
extend \ATL and have \emph{tractable} ($\mathrm{PTIME}$-complete) model checking.
The hierarchy is based on bounding the
\emph{Boolean strategic width} 
%
%
of formulae.
We denote the new fragments in the hierarchy by $\ATL^k$ for different positive integers $k$.
%
%
Here $\ATL^k$ contains those
formulae of $\ATLplus$ where subformulae $\coop{A}\Phi$ are restricted such that $\Phi$ is a
Boolean combination of at most $k$ formulae. Note that thus $\ATL^1$ corresponds to plain $\ATL$.

The current paper
extends the results in \cite{GKR-AAMAS2016}, where a \GTS for \ATL is
considered, in various non-trivial ways. 
Firstly, several new ideas and technical notions, such as the \emph{role of a seeker} and
the use of a \emph{truth function}, are introduced here in order to enable the transition
from \ATL to \ATLplus in the \GTS setting.
Secondly, a useful and generally elucidating link between our \GTS and \emph{B\"{u}chi games} is identified. That link applies readily also to the simpler evaluation games in \cite{GKR-AAMAS2016}.
Thirdly, and most importantly, we show how to use the new upgraded semantics in a model checking procedure for \ATLplus and the fragments $\ATL^k$. This would not have been possible with the semantics of \cite{GKR-AAMAS2016}.
The current paper is the journal version of \cite{GKR-AAMAS2017}. We extend
\cite{GKR-AAMAS2017} by, inter alia, including a range of new
results on systems of bounded semantics based on finite transducers.
We analyse the amount of memory resources needed for winning strategies and establish tight lower and upper bounds for it.
We notice that in transducer based semantics, an exponential amount of memory with respect to formula size is required.
However, only a linear amount of this is actually used in any concrete single evaluation process of a formula. 
Based on this we argue that the transducer based approach does not give a complete analysis for the requirement of memory resources.
%
%

\textbf{Structure of the paper.}
After the preliminaries in Section \ref{sec:prelim},  we define a bounded, finitely bounded, and unbounded game-theoretic semantics for \ATLplus in Section \ref{sec:GTS}. In Section~\ref{sec:results} we analyse the various properties of the novel systems of \GTS. In Section~\ref{sec:equivalence} we prove equivalence of the bounded and unbounded versions with the standard compositional semantics of \ATLplus with perfect recall strategies. In Section~\ref{sec:MC} we apply the \GTS to the model checking problem for \ATLplus and identify a hierarchy of tractable fragments of it.  In Section \ref{sec:BoundedMemory} we study the transducer-based bounded memory semantics for these fragments. We then conclude in Section~\ref{sec:concl}. 


\section{Preliminaries}
\label{sec:prelim}


In this section we define concurrent game models and the syntax and the (perfect-recall) semantics for \ATLplus. We also introduce some new terminology and notations that will be used later in this paper.

\begin{definition}
A \defstyle{concurrent game model} 
($\mathit{\CGM}$) 
is a tuple 
\[
\mathcal{M} := (\Agt, \St, \Prop, \Act, d, o, v)
\]
which consists of: \\ 
-- The following non-empty sets:
\defstyle{agents} $\Agt = \{a_1,\dots,a_k\}$, 
\defstyle{states} $\St$, 
\defstyle{proposition symbols} $\Prop$, 
\defstyle{actions} $\Act$; \\
-- The following functions: an \defstyle{action function} 
$d:\Agt\times\St\rightarrow\mathcal{P}(\Act)\setminus\{\emptyset\}$ 
which assigns a non-empty set of actions available to each agent at each state; a
\defstyle{transition function} $o$ which assigns an \defstyle{outcome state}
$o(q, \vec{\acta})$
to each state $q\in\St$ and action profile (a 
tuple of actions $\vec{\acta} = (\alpha_1,\dots,\alpha_k)$
such that $\alpha_i\in d(a_i,q)$ for each $a_i\in\Agt$); 
and finally, a \defstyle{valuation function} $v:\Prop\rightarrow\mathcal{P}(\St)$. 


We use symbols $p, p_0, p_{1}, \ldots$ 
to denote proposition symbols and $q, q_0, q_{1}, \ldots$  to denote states.
Sets of agents 
are called \defstyle{coalitions}. The complement  
$\overline{A}= \Agt\setminus A$ of a coalition $A$ is the \defstyle{opposing coalition} of $A$. The set $\avact(A,q)$ of action tuples \emph{available} to  coalition $A$ at state $q\in\St$ is defined as 
$\avact(A,q):=\{(\alpha_i)_{a_i\in A}\mid \alpha_i\in d(a_i,q) \text{ for each } a_i\in A\}$.
\end{definition}
\begin{example}\label{ex: model}  
Let $\mathcal{M}^* = (\Agt, \St, \Prop, \Act, d, o, v)$, where: 
\aamasp{
$\Agt = \{a_1,a_2\}, \; \St = \{q_0,q_1,q_2,q_3,q_4\}, \;  \Prop = \{p_1,p_2,p_3\}, \; \\ 
\Act = \{\alpha,\beta\}, $  
%
and $d$, $o$ and $v$ defined as shown below:
%
}
\techrep{
\begin{align*}
	&\Agt = \{a_1,a_2\}, \; \St = \{q_0,q_1,q_2,q_3,q_4\} \\
	&\Prop = \{p_1,p_2,p_3\}, \; \Act = \{\alpha,\beta\} \\
	&d(a_2,q_0)=d(a_1,q_1)=\{\alpha,\beta\} \,\text{ and else }\; d(a_i,q_i)=\{\alpha\} \\
	&o(q_0,\alpha\alpha)=q_1, \; o(q_0,\alpha\beta)=q_2, \; o(q_1,\alpha\alpha)=q_2, \\
	&o(q_1,\alpha\beta)=q_3, \; o(q_2,\alpha\alpha)=q_1 \,\text{ and }\, o(q_3,\alpha\alpha)=q_3 \\
	&v(p_1)=\{q_2,q_4\}, \; v(p_2)=\{q_3\} \;\text{ and } v(p_3)=\{q_1\}.
\end{align*}
}

\aamasp{\vspace{-0,2cm}}

\begin{center}
\begin{tikzpicture}[scale=1.8]
	\node at (-0.3,0) [draw, circle] (0) {\phantom{$p_1$}};
	\node at (1,0) [draw, circle] (1) {$p_3$};
	\node at (-0.3,1) [draw, circle] (2) {$p_1$};
	\node at (1,1) [draw, circle] (3) {$p_2$};
	\node at (2.3,0) [draw, circle] (4) {$p_1$};
	\node at (-1.5,1.3) {$\mathcal{M}^*$:};
	\node at (-0.6,-0.3) {$q_0$};
	\node at (1.3,-0.3) {$q_1$};
	\node at (-0.6,1.3) {$q_2$};
	\node at (1.3,1.3) {$q_3$};
	\node at (2.6,0.3) {$q_4$};
	\draw[-latex] (0) to node[below] {\small $\alpha\alpha$} (1);
	\draw[-latex] (0) to node[left] {\small $\alpha\beta$} (2);
	\draw[-latex] (2) to node[above] {\small $\alpha\alpha$} (3);
	\draw[-latex, bend left] (1) to node[left] {\small $\alpha\alpha$} (3);
	\draw[-latex, bend left] (3) to node[right] {\small $\alpha\alpha$} (1);
	\draw[-latex] (1) to node[below] {\small $\beta\alpha$} (4);
	\draw[-latex, loop right] (4) to node {\small $\alpha\alpha$} (4);
\end{tikzpicture}
\end{center}

\end{example}

\aamasp{\vspace{-0,5cm}}

\begin{definition}
\label{def:CGM}
Let $\mathcal{M} = (\Agt, \St, \Prop, \Act, d, o, v)$ be a \CGM.
A \defstyle{path}
in $\mathcal{M}$ is a
sequence $\Lambda:\mathbb{N}\rightarrow \St$ of states 
such that for each $n\in\mathbb{N}$, we have 
$\Lambda[n\!+\!1] = o(\Lambda[n],\vec{\acta})$ for
some admissible action profile 
$\vec{\acta}$ in $\Lambda[n]$. 
A \defstyle{finite path} (aka \defstyle{history})
is a finite prefix sequence of a path in $\mathcal{M}$.
We let $\paths(\mathcal{M})$
denote the set of all paths in $\mathcal{M}$ and $\paths_\text{fin}(\mathcal{M})$
the set of all finite paths in $\mathcal{M}$.\footnote{Note that, accordingly this terminology, a ``path'' always refers to an infinite path. We use this terminology since we mostly consider infinite paths.}

\techrep{
A \defstyle{positional strategy} of an agent $a\in\Agt$ is a function $s_a: \St\rightarrow\Act$ such that $s_a(q)\in d(a,q)$ for each $q\in\St$. 
}
A \defstyle{perfect-recall strategy}, or hereafter just  \defstyle{strategy}, of agent $a\in\Agt$ is a function  $s_a: \paths_\text{fin}(\mathcal{M})\rightarrow\Act$ such that $s_a(\lambda)\in d(a,\lambda[k])$ for each $\lambda\in\paths_\text{fin}(\mathcal{M})$ where $\lambda[k]$ is the last state in $\lambda$. 
A \defstyle{collective strategy} $S_A$ for $A\subseteq\Agt$ is a tuple of individual strategies, one for each agent in $A$. 
With $\paths(q,S_A)$ we denote the set of all paths emerging in plays beginning from $q$ where the agents in $A$ follow the strategy $S_A$.
%
%
\end{definition}
The formulae of \ATLplus are defined by the following grammar.

\smallskip

%
\emph{State formulae}: $\atla::= p \mid \neg \atla \mid \atla\vee\atla \mid \coop{A}\atlsa \quad (p\in \Prop)$

\emph{Path formulae}: 
 $\atlsa::=\atla \mid \neg\atlsa \mid \atlsa\vee\atlsa \mid \atlx \atla \mid \atla \atlu \atla$

\smallskip
\noindent
Other Boolean connectives are defined as usual,
and furthermore, $\atlf\atla$, $\atlg\atla$ and  $\atla\atlr\atlb$ are
abbreviations for $\top \atlu \atla$, $\neg(\top \atlu \neg\atla)$, and 
$\neg(\neg\atla \atlu \neg\atlb)$ respectively.
With $\Phi$ and $\Psi$ we denote path formulae only; $\varphi$, $\psi$, and $\chi$ denote any formulae.

%
\begin{definition}\label{def: compositional semantics for ATL}
Let $\mathcal{M}$ be a $\CGM$.
%
%
\defstyle{Truth} of state and path formulae of \ATLplus is
defined, respectively,
with respect to states $q\in St$ and paths $\Lambda\in\paths(\mathcal{M})$, inductively as follows, where $\varphi, \psi$ are state formulae:
%
\begin{itemize}
\item $\mathcal{M},q\models p$  iff  $q\in v(p)$\ $\mathrm{(\text{for }}p\in\Prop\mathrm{)}$.
\item $\mathcal{M},q\models \neg\varphi$ iff $\mathcal{M},q\not\models\varphi$.
\item $\mathcal{M},q\models \varphi\vee\psi$ iff $\mathcal{M},q\models\varphi$ or $\mathcal{M},q\models\psi$.
\item $\mathcal{M},q\models\coop{A} \Phi$ iff there exists a (perfect-recall) strategy $S_A$ such that $\mathcal{M},\Lambda\models\Phi$ for each $\Lambda\in\paths(q,S_A)$.
\item $\mathcal{M},\Lambda\models\varphi$
iff $\mathcal{M},\Lambda[0]\models\varphi$. 
%
\item $\mathcal{M},\Lambda\models
\X\varphi$ iff $\mathcal{M},\Lambda[1]\models\varphi$.
\item $\mathcal{M},\Lambda\models
\neg\Phi$ iff $\mathcal{M},\Lambda\not\models\Phi$.
\item $\mathcal{M},\Lambda\models\Phi\vee\Psi$ iff $\mathcal{M},\Lambda\models\Phi$ or $\mathcal{M},\Lambda\models\Psi$.
\item $\mathcal{M},\Lambda\models\varphi\U\psi$ iff
there exists $i\in\mathbb{N}$ such that $\mathcal{M},\Lambda[i]\models\psi$ and $\mathcal{M},\Lambda[j]\models\varphi$ \,for all $j < i$.
%
\end{itemize}
\end{definition}
The \defstyle{set of subformulae}, $\subf(\varphi)$, of a formula $\varphi$ is defined as usual. Subformulae with a temporal operator as the main connective will be called \defstyle{temporal subformulae}, while subformulae with $\coop{}$ as the main connective are \defstyle{strategic subformulae}. The subformula $\Psi$ of a formula 
$\varphi = \coop{A}\Psi$ is called the \defstyle{temporal objective of $\varphi$}. 
%
%
%
We also define the set $\mathit{At}(\Phi)$ of \defstyle{relative atoms} of $\Phi$ as follows:

%
\begin{itemize}
\item
$\mathit{At}(\chi\vee\chi^\prime) = \mathit{At}(\chi)\cup \mathit{At}(\chi^\prime)$
and $\mathit{At}(\neg\chi) = \mathit{At}(\chi)$.
\item
$\mathit{At}(\coop{A}\chi) = \{\coop{A}\chi\}$ and $\mathit{At}(p) = \{p\}$
for $p\in\Pi$.
\item
$\mathit{At}(\chi\atlu\chi^\prime) = \{\chi\atlu\chi^\prime\}$
and $\mathit{At}(\X\chi) = \{\X\chi\}$.
\end{itemize}

We say that $\chi\in\mathit{At}(\Phi)$ occurs \defstyle{positively} (resp. \defstyle{negatively})   \defstyle{in $\Phi$} if $\chi$ has an occurrence in the scope of an even (resp. odd) number of negations in $\Phi$.
We denote by $\mathrm{SUB}_{\mathit{At}}(\Phi)$
the subset of $\subf(\Phi)$ containing 
all relative atoms of $\Phi$ 
and also all Boolean combinations $\chi$ of
these relative atoms such that $\chi\in\mathrm{SUB}(\Phi)$.

\aamasp{\vspace{-0,1cm}}

\begin{example}\label{ex: formula}
Let 
\vcut{$\varphi^*:=\neg p_2\wedge (p_1\vee \coop{a_1}\Psi)$,} 
$\varphi^*:= \coop{a_1}\Psi$, where
\aamasp{\vspace{-0,1cm}}\[
	\Psi:= (\neg\atlx p_3\wedge\coop{a_2}\atlx p_1) \,\vee\, (\atlf p_1\wedge(\neg p_1)\atlu p_2).
\aamasp{\vspace{-0,15cm}}\]
Written without using abbreviations, $\Psi$  becomes
\aamasp{\vspace{-0,15cm}}\[
	\neg(\neg\neg\atlx p_3 \vee \neg\coop{a_2}\atlx p_1)\vee \neg(\neg(\top\atlu p_1) \vee \neg((\neg p_1)\atlu p_2)).
\aamasp{\vspace{-0,15cm}}\]
Here $\mathit{At}(\Psi)=\{\atlx p_3,\coop{a_2}\atlx p_1,\top\atlu p_1, (\neg p_1)\atlu p_2\}$, where $\coop{a_2}\atlx p_1$ is a state formula and the rest are path formulae. The formula $\atlx p_3$ occurs negatively in $\Psi$ and the rest of the formulae in $\mathit{At}(\Psi)$ occur positively in $\Psi$.
\end{example}

\section{Game-theoretic semantics}
\label{sec:GTS}

In this section we define 
\emph{bounded}, \emph{finitely bounded} 
and \emph{unbounded evaluation games} for \ATLplus.
These games give rise to three different systems of semantics, namely, the \emph{bounded}, \emph{finitely  bounded} and \emph{unbounded} $\GTS$ for \ATLplus.

These systems of semantics were defined for plain \ATL
already in \cite{GKR-AAMAS2016,GTS-ATL-2018}. The principal difference between
the bounded and unbounded \GTS is that the bounded variant \emph{forces games to 
end after a finite number of steps}. This is a significant difference achieved, as we shall see,
via requiring the players to choose ordinal numbers that can intuitively be considered to determine
upper bounds for game durations (see also Example~\ref{ex: omega branching}). In the unbounded semantics, no such ordinals are
used, and the games can continue for infinitely many rounds.

As explained in
\cite{GKR-AAMAS2016,GTS-ATL-2018}, the difference between bounded and
unbounded semantics is directly analogous to the difference between for-loops and 
while-loops. Indeed, for-loops require an extra parameter that determines the number of loop iterations, and
while-loops can possibly loop infinitely long.

Having both the bounded and unbounded semantics at our disposal will prove beneficial in
Section \ref{sec:MC} where we discuss model checking. Indeed, we shall need
the unbounded semantics for connecting fragments of \ATLplus to B\"{u}chi games and thereby obtaining
novel tractability results. On the other hand, we shall need the
bounded semantics for our proof strategy of Theorem \ref{pspace} which
confirms the PSPACE-completeness of \ATLplus model checking.

The unbounded and bounded semantics will be proved equivalent below. 
The finitely bounded semantics is not equivalent to these two. The difference 
between the finitely bounded and bounded semantics is that the parameters 
with which the players force the games to be finite are possibly infinite 
ordinals in bounded semantics and finite ordinals in finitely bounded semantics.
The finitely bounded and bounded semantics are equivalent over finite models but
not over infinite ones. The reason for introducing finitely bounded semantics is that it
provides a novel, interesting perspective on \ATL and \ATLplus while still being 
equivalent over finite (but not infinite) models with the standard semantics.

Below we shall use some terminology and notational conventions introduced in 
\cite{GKR-AAMAS2016,GTS-ATL-2018}. 

\subsection{Evaluation games: informal description}
%
Given a $\CGM$ $\mathcal{M}$,  a state $\qzero$ and a state formula $\varphi$, the 
\defstyle{evaluation game} $\evalgame(\mathcal{M}, \qzero, \varphi)$ is, intuitively, 
a formal debate between two opponents, \defstyle{Eloise} ({\bf E})  and \defstyle{Abelard} ({\bf A}), about whether the
formula $\varphi$ is true at the state $\qzero$ in the model $\mathcal{M}$. Eloise claims that $\varphi$ is true, so she  (initially) adopts the role of a \defstyle{verifier} in the game,
and Abelard tries to prove the formula false, so he is (initially) the \defstyle{falsifier}.
These roles (verifier, falsifier) can swap in the course of the game when negations are encountered in the formula. 
If $\pl\in\{{\bf E},{\bf A} \}$, 
then ${\opl}$ denotes the \defstyle{opponent} of ${\pl}$,   
i.e., ${\opl}\in \{{\bf E},{\bf A} \}\setminus\{\pl\}$.

We now provide an intuitive account of the
\emph{bounded} evaluation game and the \emph{bounded} \GTS for \ATLplus.
The intuitions underlying the finitely bounded and unbounded $\GTS$ are similar.
A reader unfamiliar with the concept of \GTS may find it useful to
consult, for example, \cite{HintikkaSandu97} for \GTS in general and 
 \cite{GKR-AAMAS2016} or \cite{GTS-ATL-2018} for \ATL-specific \GTS. 
The \GTS for \ATLplus presented here follows the 
general principles of \GTS, with the main original 
feature being the treatment of strategic formulae $\coop{A}\Phi$.
We first give an \emph{informal} account of 
the way such formulae are treated in our evaluation games. 
Formal definitions and some concrete examples will be given further, beginning from Section \ref{semanticsdfnsect}.
The evaluation of \ATLplus formulae of the type $\coop{A}\Phi$ in a given model is
based on constructing finite paths in that model. 
The following two ideas are central.

Firstly, the path formula $\Phi$ in $\coop{A}\Phi$
can be \emph{divided} into \emph{goals for the verifier} ($\vrf$), these being  
the relative atoms $\psi\in\mathit{At}(\Phi)$ that occur \emph{positively} in $\Phi$,  
and \emph{goals for the falsifier} ($\ovrf$), these being 
the relative atoms $\psi\in\mathit{At}(\Phi)$ that occur \emph{negatively} in $\Phi$.
(Some formulae may be goals for both players.) 
For simplicity, let us assume for now that $\Phi$ is in negation normal form and all the atoms in $\mathit{At}(\Phi)$  are temporal formulae of the type $\atlf p$.
Then the verifier's goals are eventuality statements $\atlf p$,
while the falsifier's goals are statements $\atlf p'$
that occur negated; note that the negation of $\atlf p'$ is
equivalent to the safety statement $\atlg \neg p'$.
%
%
The verifier wishes to \emph{verify} her/his\footnote{The genders of the
players may be assigned randomly below at points when this causes no
ambiguities and streamlines the presentation.} goals.
The falsifier, likewise, wants to \emph{verify} her/his goals, i.e.,
the falsifier wishes to \emph{falsify} the related safety statements.
%
%
%

%
Secondly, every temporal goal associated with $\coop{A}\Phi$ has a unique ``finite determination point'' 
on any given path where that goal can be verified by 
the player to whom the goal belongs. This means the following.
If a goal $\atlf p$ of the verifier is true on an infinite path  $\pi$,
then there necessarily exists an earliest point $q$
on that path where the fact that $\atlf p$ 
holds on  $\pi$ becomes \emph{verified} simply
because $p$ is true at $q$. Indeed, the first point of $\pi$ where $p$ is
true is the finite determination point $q$ of $\atlf p$.
Once $\atlf p$ has been verified, it will remain \emph{true on $\pi$}, 
no matter what happens on the path after $q$.
Similarly, concerning falsifier's goals, if $\atlg \neg p'$ is false (and thus $\atlf p'$ true) on an infinite path $\pi'$, there is a unique point where $\atlg \neg p'$ first becomes \emph{falsified}, that point being the first state $q'$ of $\pi'$
where $p'$ is true. That point $q'$ is the finite determination point of the goal $\atlf p'$ of the falsifier.
Furthermore, $\atlg \neg p'$ will
remain false on the path no matter what happens further.
(Note that there is no analogous finite determination point for $\ATLs$-formulae such as $\coop{A}\!\atlg\!\atlf\!p$ on a given infinite path. Note also that we discussed only the simple temporal goals  $\atlf p$ and $\atlf p'$
for simplicity, but every temporal goal---as long as it can be verified by 
the player to whom the goal belongs---does indeed have a finite determination point. This will
become clear below.)

Now, the game-theoretic evaluation procedure of an \ATLplus-formula 
$\coop{A}\Phi$ proceeds roughly as follows.
The verifier is controlling the agents in the coalition $A$ and the falsifier controls the agents in the \defstyle{opposing coalition} $\overline{A}=\Agt\setminus A$.
The players start constructing a path. (Each transition from one state to another is carried
out according to the process ``$\mathsf{Step\, phase}$"
defined formally in Section \ref{transitiongamerules}.)
The verifier is first
given a chance to verify some of her/his goals in $\Phi$.
The falsifier tries to prevent this and to possibly verify some of her/his own goals instead.
During this path construction/verification process, the verifier is said to have the role of the \defstyle{seeker}. 
A player is allowed to stay as the seeker for only a finite number of rounds. This is ensured by requiring the seeker to announce an ordinal\footnote{To see why finite ordinals do not suffice in 
general relates to infinite branching. See, e.g., Example 3.11 of \cite{GTS-ATL-2018} for details.}, 
called \defstyle{timer}\footnote{\small Note that the term ``timer'' is used here differently from \cite{GKR-AAMAS2016, GTS-ATL-2018}.
}, 
before the path construction process begins, and then lower the ordinal each time a new state is reached. 
The process ends when the ordinal becomes zero or when the seeker 
is satisfied, having verified some of her goals.
Since ordinals are well-founded, the process must terminate.

After the verifier has ended her/his seeker turn, the falsifier may either end the game or take the role of the seeker.
If (s)he decides to become the seeker, then (s)he sets a new
timer and the path construction 
process continues for some finite number of rounds. 
When the falsifier is satisfied,
having verified some of her/his goals,
the verifier may again take the seeker's role, and so on. 
Thus, the verifier and falsifier take turns 
being the seeker, trying to reach (verify) their goals.
The number of these alternations is bounded by a
\defstyle{seeker turn counter}
which is a finite number that equals the total number of goals in $\Phi$.
(The formal description of seeker turn alternation is
given in the clause ``$\mathsf{Deciding\, whether\, to}$
$\mathsf{continue\, and\, adjusting\, the\, timer}$" in Section \ref{transitiongamerules}.)

Each time a goal in $\Phi$ becomes verified, this is recorded in a \defstyle{truth function} $T$.
(The recording of verified goals is described formally in the
process ``$\mathsf{Adjusting\, the\, truth\, function}$" defined
in Section \ref{transitiongamerules}.)
The truth function carries the following information at any stage of the game:
\begin{itemize}
\item
The verifier's goals that have been verified. 
\item
The falsifier's goals that have been verified. 
\item All other goals remain \defstyle{open}.
\end{itemize}
When neither of the players wants to become the seeker, or when the seeker turn counter becomes zero, the path construction process \emph{ends} and the 
players play a standard Boolean evaluation game on $\Phi$ by using the values given by $T$; the open goals are given truth values as follows:
\begin{itemize}
\item
The verifier's open goals are \emph{(so far) not verified} and thus considered \emph{false}.
\item
Likewise, the falsifier's open goals are \emph{(so far) not verified} and thus considered \emph{false}. Recall  here that the falsifier's goals occur in the scope of a negation.
\end{itemize}
Next we consider the conditions when a player is ``satisfied'' with the current status of the truth function $T$---and thus wants to end the game---and when (s)he is ``unsatisfied'' and wants to continue the game as the seeker.
Note that when the path construction ends, then every goal is
given a Boolean truth value based on the truth function $T$, as
described above. With these values, the formula $\Phi$ is either true or false.
If $\Phi$ is true with the current values based on $T$, then the verifier can win the Boolean game for $\Phi$; dually, if $\Phi$ is not true with the values based on $T$, then the falsifier can win the Boolean game for $\Phi$. Hence the players want to take the role of the seeker in order to modify the truth function $T$ in such a way that the truth of $\Phi$ with respect to $T$ changes from false to true (whence $\vrf$ is satisfied) or from true to false (whence $\ovrf$ is satisfied).

The truth value of $\Phi$ with respect to $T$
can keep changing when $T$ is modified, but only a finite number of changes is possible. 
Indeed, the maximum number of such truth alternations is the total number of goals in $\Phi$.
%

\techrep{
\medskip

In the general case, formulae of the type $\varphi\atlu\psi$, $\atlx\varphi$ and (state formulae) $\varphi$ may also occur in $\mathit{At}(\Phi)$ as goals, and $\Phi$ does not have to be in negation normal form. Formulae of the type $\varphi\atlu\psi$ can be either verified, by showing that $\psi$ is true, or falsified, by showing that $\varphi$ is not true at related states. State formulae $\varphi$ can only be verified at the initial state and the next-state-formulae $\atlx\varphi$ can only be verified at the second state on the path traveled.
%
}

\aamasp{\vspace{-0,1cm}}

\subsection{Evaluation games: formal description}
\label{semanticsdfnsect}

Now we will present the \defstyle{bounded evaluation game} which uses the \defstyle{bounded transition game} as a subgame for evaluating strategic subformulae. 
Interleaved with the definition we will provide, in \emph{italics}, a running example that uses
$\mathcal{M}^*$ and $\varphi^*$ from Examples~\ref{ex: model} and~\ref{ex: formula}
respectively.
%

\subsubsection{Rules of the bounded evaluation~game} 

Let $\mathcal{M} = (\Agt, \St, \Prop, \Act, d, o, v)$ be a \CGM, $\qzero\!\in\!\St$ a state, $\varphi$ a state formula and $\Gamma>0$ an ordinal called a \defstyle{timer bound}.
The \defstyle{$\tlimbound$-bounded evaluation~game} 
$\evalgame(\mathcal{M}, \qzero, \varphi,\Gamma)$ 
between the players \defstyle{\bf A} and \defstyle{\bf E}  is defined as follows.

A \defstyle{location} of the game is a tuple $({\pl},q,\psi,T)$ where 
${\pl}\in\{\text{\bf A},\text{\bf E}\}$,
$q\in\St$ is a state, $\psi$ is a subformula of $\varphi$
and $T$ is a \defstyle{truth (history) function}, mapping some subset of $\subf(\varphi)$ into $\{\top,\bot,\open\}$.\footnote{We note here that the values of $T$ are only modified during transtion games and that $T$ is always a total function for all subformulae of $\varphi$ that are relevant for the transition game that is played.}

The \defstyle{initial location} of the
game is $(\text{\bf E},\qzero,\varphi,T_{in})$,
where $T_{in}$ is the empty function.
%
In every location $({\pl},q,\psi,T)$, the player \pl\ is called the \defstyle{verifier}
and \opl\ the \defstyle{falsifier} for that location.
Intuitively, $q$ is the current state of the game and $T$ encodes truth values of formulae on a path that has been constructed earlier in the game.
Each location is associated with exactly one of the rules \defstyle{1--6} given below.
First we provide the rules for locations $({\pl},q,\psi,T)$ where $\psi$ is
either a proposition symbol or has a Boolean connective as its
main operator:

\begin{enumerate}
\item[\bf 1.]
A location $({\pl},q,p,T)$, where $p\in\Prop$,
is an \defstyle{ending location} of the evaluation game.
If $T \neq \emptyset$, 
then $\pl$ wins the game if $T(p) = \top$ and else $\opl$ wins.
Respectively, if $T = \emptyset$, 
then ${\pl}$ wins if $q\in v(p)$ and else ${\opl}$ wins.

\item[\bf 2.] From a location $({\pl},q,\neg\psi,T)$
the game moves to the location $({\opl},q,\psi,T)$.

\item[\bf 3.] In a location $({\pl},q,\psi\vee\theta,T)$ the player ${\pl}$
chooses one of the locations $({\pl},q,\psi,T)$ and $({\pl},q,\theta,T)$,
which becomes the next location of the game.
\end{enumerate}

We then define the rules of the evaluation game for
locations with strategic formulae as follows.
%

\begin{enumerate}
\item[\bf 4.] Suppose a location $(\pl,q,\coop{A}\Phi,T)$ is reached.
\begin{itemize}[leftmargin=*]
\item If 
$T \neq \emptyset$, then this location is an ending location where $\pl$ wins if $T(\coop{A}\Phi)=\top$ and else $\opl$ wins. 

\item If 
$T  = \emptyset$, 
then the evaluation game enters a \defstyle{transition game}
$\defstyle{g}(\pl,q,\coop{A}\Phi,\tlimbound)$. The transition game is a
subgame to be defined later on.
The transition game eventually reaches an \defstyle{exit location} $(\pl',q',\psi,T')$,
and the evaluation game continues from that location. Note that an \emph{exit location} only
ends the transition game, so exit locations of transition games and
\emph{ending locations} of the evaluation game are different concepts.
\end{itemize}
\end{enumerate}

The rules corresponding to the temporal connectives are defined using the truth function $T$
(updated in an earlier transition game) as follows.

\begin{enumerate}
\item[\bf 5.]
A location $(\pl,q,\varphi\atlu\psi,T)$ is an ending location of the evaluation game. \\
$\pl$ wins if $T(\varphi\atlu\psi) =\top$ and else $\opl$ wins.

\item[\bf 6.]
Likewise, a location $(\pl,q,\X\varphi,T)$ is an ending location. \\ 
$\pl$ wins if $T(\X\varphi) = \top$ and otherwise $\opl$ wins.
\end{enumerate}

These are the rules of the evaluation game. We note that the timer bound
$\tlimbound$ will be used only in transition games. If $\tlimbound=\omega$, we say that the evaluation game is \defstyle{finitely bounded}. 

\smallskip

\runex{
The  initial location of the finitely bounded evaluation game 
$\mathcal{G}(\mathcal{M}^*,q_0,\varphi^*,\omega)$ 
(see Examples~\ref{ex: model} and \ref{ex: formula}) is $({\bf E},q_0,\coop{a_1}\Psi,\emptyset)$, from where the transition game 
$\defstyle{g}({\bf E},q_0,\coop{a_1}\Psi,\omega)$ begins.
}

\vcut{
\runex{
Consider the evaluation game $\mathcal{G^*}:=\mathcal{G}(\mathcal{M}^*,q_0,\varphi^*,\omega)$ (see Examples~\ref{ex: model} and \ref{ex: formula}). The game begins from the initial location $({\bf E},q_0,\neg p_2\wedge (p_1\vee\coop{a_1}\Psi),\emptyset)$. It is easy to see that the rule for the conjunction is that the falsifier gets to choose either of the conjuncts. So, here Abelard may first choose the next location of the game to be $({\bf E},q_0,\neg p_2,\emptyset)$ or $({\bf E},q_0,p_1\vee\coop{a_1}\Psi,\emptyset)$. If Abelard chooses the first option, then the next location is $({\bf A},q_0,p_2,\emptyset)$, where Abelard loses, since $q_0\notin v(p_2)$. 
Suppose now that Abelard chooses the second option. Now Eloise gets to  choose the next location of the game to be $({\bf E},q_0,p_1,\emptyset)$ or $({\bf E},q_0,\coop{a_1}\Psi,\emptyset)$. If Eloise would choose the first, she would lose immediately since $q_0\notin v(p_1)$. So, suppose Eloise chooses the second option. Then the transition game $\defstyle{g}({\bf E},q_0,\coop{a_1}\Psi,\omega)$ begins.
}
}

\subsubsection{Rules of the transition game}\label{transitiongamerules}

\techrep{Recall that transition games are subgames of evaluation games. Their purpose is to evaluate the truth of strategic subformulae, in a game-like fashion.}  

Now we give a detailed description of transition games.\techrep{\footnote{\small A transition game for $\ATLplus$ is similar to the `embedded game'  introduced in \cite{GKR-AAMAS2016,GTS-ATL-2018} for the \GTS of \ATL. The role of the seeker $\ctr$ here is similar to the role of the controller in that embedded game.}}
A \defstyle{transition game}
$\defstyle{g}(\vrf,q_0,\coop{A}\Phi,\tlimbound)$, 
where ${\vrf}\in\{\text{\bf A},\text{\bf E}\}$,
$q_0\in\St$, $\coop{A}\Phi\in\ATLplus$ and $\tlimbound>0$ is an ordinal,
is defined as follows.
$\vrf$ is called \defstyle{the verifier in the transition game}. 
The game $\defstyle{g}(\vrf,q_0,\coop{A}\Phi,\Gamma)$
is based on \defstyle{configurations}, i.e.,
tuples $(\ctr,q,T,n,\gamma,x)$, where  
the player $\ctr\in\{\mathbf{E},\mathbf{A}\}$ is called the \defstyle{seeker}; 
$q$ is the \defstyle{current state}; 
$T:\mathit{At}(\Phi)\rightarrow\{\top,\bot,\mathsf{open}\}$ is a 
\defstyle{truth function}; 
$n\in\mathbb{N}$ is a \defstyle{seeker turn counter} ($n\leq |\mathit{At}(\Phi)|$);
$\gamma$ is an ordinal called \defstyle{timer};
and $x\in\{\, \defstyle{i},\defstyle{ii},\defstyle{iii}\, \}$ is
an index showing the current \defstyle{phase} of the transition game.
The game\! $\defstyle{g}(\vrf\!,q_0,\!\coop{A}\Phi,\tlimbound)$\! begins\!
at\! the \defstyle{initial configuration} 
$(\vrf,q_0,T_0,|\mathit{At}(\Phi)|,\tlimbound,\defstyle{i})$, 
with $T_0(\chi) = \open$ for all $\chi\in\mathit{At}(\Phi)$.

\smallskip

\runex{
The transition game $\defstyle{g}({\bf E},q_0,\coop{a_1}\Psi,\omega)$ begins from the initial configuration $({\bf E},q_0,T_0,4,\omega,\emph{\defstyle{i}})$, since $|\mathit{At}(\Psi)|=4$. (Note that the timer is initially $\omega$ in transition games occurring within finitely bounded evaluation games, but the timer will always have a finite value thereafter.)}

\smallskip

The transition game then proceeds by iterating the phases \defstyle{i}, \defstyle{ii} and \defstyle{iii}, which we {first} describe informally; detailed formal definitions are given afterwards.
\begin{itemize}[leftmargin=6mm]
\item[\defstyle{i.}] $\mathsf{Adjusting\, the\, truth\, function}$:
In this phase\ the\hspace{1mm}  players\hspace{1mm} make claims on the truth of state formulae at the current state $q$. If $\pl$ makes some claim, then the opponent $\opl$ may either: 1) accept the claim, whence truth function is updated accordingly, or 2) challenge the claim. In the latter case the \emph{transition game ends} 
and truth of the claim is verified in a continued evaluation game.

\item[\defstyle{ii.}] $\mathsf{Deciding\, whether\, to\, continue\, and\, adjusting\, the\, timer}$:
Here the current seeker $\ctr$ may either continue her seeker turn and lower the value of the timer, or end her seeker turn. If $\ctr$ chooses the latter option, then the opponent $\octr$ of the seeker may either 1) take the role of the seeker and announce a new value for the timer or 2) end the transition game, whence the formula $\Phi$ is evaluated based on current values of the truth function.

\item[\defstyle{iii.}] $\mathsf{Step\, phase}$:
Here the verifier $\vrf$ chooses actions for the agents in the coalition in $A$ at the current state $q$. Then $\ovrf$ chooses actions for the agents in the opposing coalition $\overline{A}$. After the resulting transition to a new state $q'$ has been made, the game continues again with phase \defstyle{i}.
\end{itemize}
We now describe the phases \defstyle{i}, \defstyle{ii} and \defstyle{iii} in technical detail:

\bigskip

\noindent
\defstyle{i.} $\mathsf{Adjusting\, the\, truth\, function.}$

Suppose the current configuration is $(\ctr,q,T,n,\gamma,\defstyle{i})$.
Then the truth function $T$ is updated by considering, one by one, each formula $\chi\in\mathit{At}(\Phi)$ in some fixed order\footnote{We will see that the order here is irrelevant for the existence of winning strategies in the evaluation game. This is simply because the player with a winning strategy can make all the claims that are true and oppose all the other claims---regardless of the order in which the formulae are considered.}. 
If $T(\chi)\neq\open$, then the value $\chi$ cannot be updated. Else the value of $\chi$ may be modified according to the rules \defstyle{A} -- \defstyle{C} below.

\smallskip

\defstyle{A.} \emph{Updating $T$ on temporal formulae with $\U$}: Suppose
that $\varphi\atlu\psi\, \in\, At(\Phi)$.
Now first the verifier $\vrf$ may \emph{claim that $\psi$ is true} at the current state $q$.
If $\vrf$ makes that claim, then $\ovrf$ chooses either of the following: 
\begin{itemize}
\item
$\ovrf$ \emph{accepts} the claim of $\vrf$, whence the truth function is updated so that 
$\varphi\atlu\psi$ is assigned value $\top$ ($\varphi\atlu\psi$ becomes \defstyle{verified}), hereafter indicated by $\varphi\atlu\psi\, \mapsto\top$.
\item
$\ovrf$ \emph{challenges} the claim of $\vrf$, whence the transition game ends
at the \defstyle{exit location} $(\vrf,q,\psi,\emptyset)$. 
(We note that, here and further, when a transition game ends, the evaluation game 
continues from the related exit location and the evaluation game will \emph{never}
return to the same exited transition game again.)
\end{itemize}
If $\vrf$ does not claim that $\psi$ is true at $q$,
then $\ovrf$ may make that same claim (that $\psi$ is true at $q$).
If $\ovrf$ makes that claim, then the same two steps above concerning
\emph{accepting} and \emph{challenging} 
are followed, but with $\vrf$ and $\ovrf$ swapped everywhere.

Suppose then that neither of the players claims that $\psi$ is true at $q$. 
Then first $\vrf$ can \emph{claim that $\varphi$ is false} at $q$. 
If $\vrf$ makes that claim, then $\ovrf$ chooses either of the following: 

\begin{itemize}
\item $\ovrf$ accepts the claim, whence the truth function is updated so that
$\varphi\atlu\psi \mapsto \bot$ ($\varphi\atlu\psi$ becomes \defstyle{falsified}).
\item $\ovrf$ challenges the claim, whence the transition game ends
at the exit location $(\ovrf,q,\varphi,\emptyset)$.
\end{itemize}
If $\vrf$ does not claim that $\varphi$ is false at $q$, then $\ovrf$ may make that claim.
If he does, then the same steps as those above are followed, but with $\vrf$ and $\ovrf$ swapped.

\smallskip

\defstyle{B.} \emph{Updating $T$ on proposition symbols and strategic formulae}: The
truth function can be updated on proposition symbols $p\in\mathit{At}(\Phi)$ and formulae $\coop{A'}\Psi\in\mathit{At}(\Phi)$ only when the phase \defstyle{i} is executed for the first time (so, $q = q_0$). 
In this case, given such a formula $\chi$, first $\vrf$ can claim that $\chi$
%
is true at $q$.
Now, if $\ovrf$ accepts this claim, 
then the truth function is updated s.t. $\chi \mapsto \top$. 
If $\ovrf$ challenges the claim, 
then the transition game ends at the exit location $(\vrf,q,\chi,\emptyset)$.
If $\vrf$ does not claim that $\chi$ is true at $q$, then $\ovrf$ may make that claim.
If he does, then the same steps are followed, but with $\vrf$ and $\ovrf$ swapped.

\smallskip

\defstyle{C.} \emph{Updating $T$ on formulae with $\atlx$}: The
truth function can be updated on formulae of type $\X\psi\in\mathit{At}(\Phi)$ only when phase \defstyle{i} is executed for the second time in the transition game
(so, $q$ is some successor of $q_0$).
First $\vrf$ can claim that $\psi$ is true at $q$. 
If $\ovrf$ accepts that claim, 
then the truth function is updated s.t. 
$\atlx\psi \mapsto \top$.
If $\ovrf$ challenges the claim, 
then the transition game ends at the exit location $(\vrf,q,\psi,\emptyset)$.
If $\vrf$ does not claim that $\psi$ is true at $q$, then $\ovrf$ can make that claim.
If he does, the same steps are followed, but with $\vrf$ and $\ovrf$ swapped.

\smallskip

Note that in points \defstyle{B} and \defstyle{C}, the formulae cannot be mapped to $\bot$ by the truth function $T$. But if these formulae are left with the value $\open$, then they will be considered false by default if the transition game ends in stage \defstyle{ii} (and the boolean game is played). Intuitively this is because if no player has claimed these formulae to be true, then players have agreed that they are indeed false.

\smallskip

If neither player makes any claim which would update
the value of a formula $\chi\in\mathit{At}(\Phi)$, then the value of $\chi$ is left $\open$. Once the values of the truth function $T$ have been updated (or left as they are) for all formulae in $\mathit{At}(\Phi)$, a new truth function $T'$ is obtained. 
The transition game then moves to the new configuration $(\ctr,q,T^\prime,n,\gamma,\defstyle{ii})$. 

\smallskip

\runex{ 
In the configuration $({\bf E},q_0,T_0,4,\omega,\emph{\defstyle{i}})$ the players begin adjusting $T_0$ for which initially $T_0(\chi)=\open$ for every $\chi\in\mathit{At}(\Psi)$. Since it is the first round of the transition game, the value of $\atlx p_3$ cannot be modified, but the value of $\coop{a_2}\atlx p_1$ can be modified. 
Suppose that Eloise claims that $\coop{a_2}\atlx p_1$ is true at $q_0$. Now Abelard could challenge the claim, whence the transition game ends and the evaluation game continues from location $({\bf E},q_0,\coop{a_2}\atlx p_1,\emptyset)$ (which leads to a new transition game $\defstyle{g}({\bf E},q_0,\coop{a_2}\atlx p_1,\omega)$). Suppose Abelard does not challenge the claim. Then $\coop{a_2}\atlx p_1$ is mapped to $\top$.
}

\runex{ 
Since $\atlf p_1$ and $(\neg p_1)\atlu p_2$ occur positively in $\Phi$, Eloise has interest only to verify them and Abelard has interest only to falsify them. Eloise could verify $\atlf p_1$ by claiming that $p_1$ is true, or verify $(\neg p_1)\atlu p_2$ by claiming that $p_2$ is true. But,  if Eloise makes either of these claims, then Abelard wins the whole evaluation game by challenging, since $q_0\notin v(p_1)\cup v(p_2)$. Suppose that Eloise does not make any claims. 
Now, Abelard could claim that $\neg p_1$ is not true, in order to falsify $(\neg p_1)\atlu p_2$. But if he does that, he loses the evaluation game if Eloise challenges, since $q_0\notin v(p_1)$. Suppose that Abelard does not make any claims either. Then the transition game proceeds to configuration $({\bf E},q_0,T,4,\omega,\emph{\defstyle{ii}})$, where $T(\coop{a_2}\atlx p_1)=\top$ and $T(\chi)=\open$ for the other $\chi\in\mathit{At}(\Psi)$.
}

\bigskip

\noindent
$\defstyle{ii.}\ \mathsf{Deciding\, whether\, to\, continue\, and\, adjusting\, the\, timer}$.

Suppose a configuration $(\ctr,q,T,n,\gamma,\defstyle{ii})$ has been reached.
Assume first that $\gamma \not= 0$. Then the seeker $\ctr$ can choose whether to continue the transition game as the seeker.
If yes, then $\ctr$ chooses some ordinal $\gamma'<\gamma$ and the transition game continues from $(\ctr,q,T,n,\gamma',\defstyle{iii})$.
If $\ctr$ does not want to continue, or if $\gamma = 0$,
then one of the following applies.
\begin{enumerate}
\item[(a)]
Suppose that $n\not= 0$. Then the player $\octr$ 
chooses whether she wishes to continue the transition game.
If yes, then $\octr$ chooses an ordinal $\gamma'<\tlimbound$ 
(so, $\octr$ in fact \emph{resets} the timer value)
and the transition game continues from $(\octr,q,T,n-1,\gamma',\defstyle{iii})$.
Otherwise the transition game ends
at the exit location $(\vrf,q,\Phi,T)$.
\item[(b)]
Suppose that $n=0$. Then the transition game ends
at the exit location $(\vrf,q,\Phi,T)$.
\end{enumerate}

\runex{ 
In $({\bf E},q_0,T,4,\omega,\emph{\defstyle{ii}})$ Eloise may decide whether to continue the transition game as the seeker. Suppose that Eloise does not continue, whence Abelard may now become the seeker and continue the transition game, or end it. If Abelard ends the transition game, then the evaluation game is continued from $({\bf E},q_0,\Psi,T)$. But because $T(\atlx p_3)=\open$ and $T(\coop{a_2}\atlx p_1)=\top$, Eloise can then win the evaluation game by choosing the left disjunct of $\Psi$ (recall that with these values of $T$ Eloise is then guaranteed to win). Suppose thus that Abelard decides to become the seeker, whence he chooses some $m<\omega$ and the next configuration is $({\bf A},q_0,T,3,m,\emph{\defstyle{iii}})$.
}

\bigskip

\noindent
$\defstyle{iii.}\ \mathsf{Step\, phase}$
\footnote{\small The procedure in this phase is analogous to the \emph{step game}, 
$\step(\vrf,A,q)$, which was introduced for the \GTS for \ATL (\cite{GKR-AAMAS2016,GTS-ATL-2018}).}

Suppose that the configuration is 
$(\ctr,q,T,n,\gamma,\defstyle{iii})$.
\begin{enumerate}
\item[(a)] First, ${\vrf}$ chooses an action $\alpha_i\in d(a_i,q)$ for each $a_i\in A$.
\item[(b)] Then, ${\ovrf}$ chooses an action $\alpha_i\in d(a_i,q)$ for each $a_i\in\overline{A}$.
\end{enumerate}
The resulting action profile produces a \defstyle{successor state} $q':=o(q,\alpha_1,\dots,\alpha_k)$. 
%
%
The transition game then moves to the configuration 
$(\ctr,q',T,n,\gamma,\defstyle{i})$.

\smallskip

\runex{ 
In the configuration $({\bf A},q_0,T,3,m,\emph{\defstyle{iii}})$ Eloise (who is the verifier $\vrf$) first chooses action for agent $a_1$, then Abelard chooses action for agent $a_2$, which produces either successor state  $q_1$ or $q_2$. Then the transition game continues from the configuration $({\bf A},q_j,T,3,m,\emph{\defstyle{i}})$, where $j\in\{1,2\}$.
}

\smallskip

This concludes the definition of the rules for the phases \defstyle{i, ii} and \defstyle{iii} in the transition game $\defstyle{g}(\vrf,q_0,\coop{A}\Phi,\tlimbound)$.

\smallskip

\runex{ 
\aamasp{
Suppose first that the transition game is continued from $({\bf A},q_2,T,3,m,\emph{\defstyle{i}})$. 
Since it is the second round, Abelard could now try to verify $\atlx p_3$ by claiming that $p_3$ is true at $q_2$. However, then Eloise would win by challenging. But if Abelard does not try to verify $\atlx p_3$ now, then the value of $\atlx p_3$ will stay $\open$. In that case Eloise will win the evaluation game simply by not making any more claims in the transition game.
}
}

\runex{ 
\aamasp{
Suppose then that the game continues from $({\bf A},q_1,T,3,m,\emph{\defstyle{i}})$. 
Suppose that Abelard  verifies $\atlx p_3$ by claiming that $p_3$ is true and that Eloise does not challenge. If the transition game now ended at $({\bf E},q_1,\Psi,T')$ with $T'(\atlx p_3)=\top$, Abelard would win. Thus, suppose that Abelard ends his seeker turn and Eloise chooses some finite timer, say $2$. At $({\bf E},q_1,T',2,2,\emph{\defstyle{iii}})$ Eloise can force the resulting state $q_3$ by choosing $\alpha$ for $a_1$. At $({\bf E},q_3,T',2,2,\emph{\defstyle{i}})$ Eloise can verify $(\neg p_1)\atlu p_2$ by claiming that $p_2$ is true at $q_3$. Furthermore, Eloise can move via $q_1$ to $q_4$ and verify $\atlf p_1$ there, before timer reaches $0$. When the evaluation game is eventually continued, Eloise wins by choosing the right disjunct of $\Psi$.
}
}

\techrep{\vspace{-12mm}}

\runex{
\techrep{
Suppose that the transition game continues from the configuration $({\bf A},q_2,T',3,m,\defstyle{i})$. Since it is the second round of the transition game, Abelard could now try to verify $\atlx p_3$ by claiming that $p_3$ is true at $q_2$. However, then Eloise could win by challenging this claim. But if Abelard does not try to verify $\atlx p_3$ at that configuration, then the value of $\atlx p_3$ will stay $\open$. Hence, when Abelard decides to end his seeker's  turn or when the timer $m$ is lowered to $0$, then Eloise may end the transition game and win the evaluation game from a location of the form $({\bf E},\Psi,q',T'')$.
}
}

\runex{
\techrep{
Suppose now that the transition game continues from the configuration $({\bf A},q_1,T',3,m,\defstyle{i})$. Suppose that Abelard verifies $\atlx p_3$ by claiming that $p_3$ is true and that Eloise does not challenge that claim. If the transition game now ended at location $({\bf E},q_1,\Psi,T'')$, where $T''(\atlx p_3)=\top$, Abelard would win. Thus, if Abelard decides to quit the transition game, then Eloise wants to continue as a seeker from configuration $({\bf E},q_1,T'',2,m',\defstyle{iii})$ for some $m'<\omega$. Then Eloise can choose action $\alpha$ for agent $a_1$ and lower the timer to $2$, whence the next configuration is $({\bf E},q_3,T'',2,2,\defstyle{i})$. Eloise can then verify $(\neg p_1)\atlu p_2$ at it by claiming that $p_2$ is true at $q_3$.
Furthermore, Eloise can move via $q_1$ to $q_4$ and verify $\atlf p_1$ there, before the timer reaches $0$. Then Eloise will win when the evaluation game is continued from a location of the form $({\bf E}, q_4,\Psi,T''')$.
}
}

\aamasp{\vspace{-0,8cm}}


\subsubsection{The unbounded evaluation game}
\label{subsec:unbounded}

Let $\evalgame(\mathcal{M}, q, \varphi,\Gamma)$  be a $\tlimbound$-bounded evaluation game. We can define a corresponding \defstyle{unbounded evaluation game}, $\evalgame(\mathcal{M}, q, \varphi)$, by replacing transition games $\defstyle{g}(\vrf,q,\coop{A}\Phi,\tlimbound)$ with \defstyle{unbounded transition games}, $\defstyle{g}(\vrf,q,\coop{A}\Phi)$; these are played with the same rules as $\defstyle{g}(\pl,q_0,\coop{A}\Phi,\Gamma)$  except that timers $\gamma$ are not used in them. Instead, the players can keep the role of a seeker for arbitrarily long and thus the game may last for an infinite number of rounds. In the case of an infinite play, the player who took the last seeker turn loses the entire evaluation game. (Recall that the number of seeker alternations is bounded by the number $|\mathit{At}(\Phi)|$.) 

\aamasp{\vspace{-0,1cm}}


\subsection{Defining the game theoretic semantics}

In this section we define game-theoretic semantics for \ATLplus by equating truth of formulae with the existence of a winning strategy for Eloise in the corresponding evaluation game.
We begin with the following remark which will be relevant for the notion of positional strategies in evaluation games.

\begin{remark}\label{remark}
The description of transition games above is based on a simplified notion of configurations. The phases \defstyle{i}--\defstyle{iii} consist of several ``subphases'' and more information should be encoded into configurations.  The full notion of configuration  should also include:

-- In phase \defstyle{i}, a counter indicating the relative atom currently under consideration by the players; flags for each player indicating whether and what claim (s)he has made on the truth of the current relative atom; a 3-bit flag indicating if it is the first, second, or some later round in the transition game. 

-- For phase \defstyle{ii}, a flag whether the current seeker wants to continue,
and for phase \defstyle{iii}, a record of the
current choice of actions for the agents in $A$ by ${\vrf}$.
%

For technical simplicity, we omit these formal details.
\end{remark}

Hereafter a \defstyle{position} in an evaluation game will mean either a location of the form $(\pl,q,\varphi,T)$ or a configuration in the fully extended form described in the remark above. By this definition, at every position only one of the players (Abelard or Eloise) has a move to choose. Thus, the entire evaluation game---including transition games as subgames---is a turn-based game of perfect information. 

\techrep{
By \defstyle{game tree} $T_{\mathcal{G}}$ of an evaluation game $\mathcal{G}$, we mean the tree whose nodes   correspond to all positions arising in $\mathcal{G}$, and every branch of which corresponds to a possible play of $\mathcal{G}$ (including transition games as subgames). Note that some of these plays may be infinite, but only because an embedded transition game does not terminate, in which case a winner in the entire evaluation game is uniquely assigned according to the rules in 
Section \ref{subsec:unbounded}.}

The formal definitions of players' memory-based strategies in the evaluation games games are defined as expected, based on histories of positions. As usual, a strategy for a player $\pl$ is called  \defstyle{winning} if, following that strategy, $\pl$ is guaranteed to win regardless of how $\opl$ plays. 
A strategy is \defstyle{positional} if it depends only on the current position.
\techrep{We can also define strategies for transition games that arise within evaluation games; note that these are substrategies for the strategies in evaluation games. A strategy $\tau$ for a transition game is called winning for $\pl$ if
\begin{itemize}[itemsep=-2mm]
\item
every exit location that can be reached with $\tau$ is a winning location for $\pl$ in the evaluation game that continues from the exit location, and additionally,
\item in the alternative scenario where the transition game continues infinitely long while $\tau$ is followed (which is possible only in unbounded games), the player $\pl$ is \emph{\textbf{not}} the player who holds the (necessarily last) seeker's turn that lasts infinitely long.
\end{itemize}
}

\aamasp{\vspace{-0,2cm}}

\begin{definition}\label{def: GTS}
Let $\mathcal{M}$ be a \CGM, $q\in\St$, $\varphi\in\ATLplus$ and $\tlimbound$ an ordinal. 
\defstyle{Truth of $\varphi$ in the $\tlimbound$-bounded} 
$(\Vdash_\tlimbound)$, resp. \defstyle{unbounded  
$(\Vdash)$ \GTS} is defined as follows: 
\aamasp{\vspace{-0,1cm}}\begin{align*}
	&\text{$\mathcal{M},q\Vdash_\tlimbound\varphi$ 
		(resp. $\mathcal{M},q\Vdash\varphi$) iff Eloise has a positional} \\[-0,1cm]
	&\hspace{0,5cm}\text{winning strategy in $\mathcal{G}(\mathcal{M},q,\varphi,\tlimbound)$ 
		(resp. $\mathcal{G}(\mathcal{M},q,\varphi)$).}
\end{align*}
\end{definition}

\techrep{
We will show later that evaluation games are determined with positional strategies. Hence, if we allowed perfect-recall strategies in the truth definition above, we
would obtain equivalent semantics.
}

\aamasp{\vspace{-0,35cm}}

\begin{example} 
\label{ex:2} 
Consider the \CGM \ 
\techrep{	
$\mathcal{M} = (\Agt, \St, \Prop, \Act, d, o, v)$, where:  
\begin{align*}
	&\Agt = \{1,2\}, \; \St = \{q_0, q_1, q_2\}, \; \Prop = \{p_1,p_2\}, \; 
	\Act = \{\alpha, \beta\} \\
	& d(1,q_0)= d(2,q_1) = \{\alpha, \beta\}; \; d(a,q_i)=\{\alpha\}  \; 
	\mbox{in all other cases;} \\ 
	&o(q_0,\beta\alpha)=q_0, \; o(q_0,\alpha\alpha)= o(q_1,\alpha\beta)=q_1, \; 
		o(q_1,\alpha\alpha)=o(q_2,\alpha\alpha)=q_2 \\
	&v(p_1)=\{q_0\} \;\text{ and } v(p_2)=\{q_2\}.
\end{align*}
	}

\aamasp{\vspace{-2.1mm}}
\begin{center}
\begin{tikzpicture}[scale=1.5]
	\node at (-0.3,0) [draw, circle] (0) {$p_1$};
	\node at (1,0) [draw, circle] (1) {\phantom{$p_3$}};
	\node at (2.3,0) [draw, circle] (4) {$p_2$};
	\node at (-1.8,0) {$\mathcal{M}:$};
	\node at (-0.3,-0.45) {$q_0$};
	\node at (1.05,-0.45) {$q_1$};
	\node at (2.35,-0.45) {$q_2$};
	\node at (1.35,0.5) {\small$\alpha\beta$};
	\draw[-latex, loop left] (0) to node {\small $\beta\alpha$} (0);	
	\draw[-latex] (0) to node[below] {\small $\alpha\alpha$} (1);
	\draw[-latex, loop above] (1) to (1);		
	\draw[-latex] (1) to node[below] {\small $\alpha\alpha$} (4);
	\draw[-latex, loop right] (4) to node {\small $\alpha\alpha$} (4);
\end{tikzpicture}
\end{center}
\aamasp{\vspace{-2.5mm}}
\aamasp{
\runex{
\hspace{-0,23cm}
Here we have $\mathcal{M} = (\Agt, \St, \Prop, \Act, d, o, v)$, where  
$\Agt = \{a_1,a_2\}$, $\Prop = \{p_1,p_2\}$,
$\St = \{q_0, q_1, q_2\}$, $\Act = \{\alpha, \beta\}$, and 
the transition, outcome and valuation functions are defined as above.} 
}

\runex{ 
Let $\varphi:=\coop{a_2}(\atlg p_1 \vee\, \atlf p_2)$ (here $\atlg p_1 = \neg\atlf\neg p_1$). We describe a winning strategy for Eloise in the unbounded evaluation game $\evalgame(\mathcal{M},q_0,\varphi)$.
Eloise immediately ends her seeker's turn and does not make claims while being at $q_0$. If Abelard makes claims at $q_0$, she challenges those claims. If Abelard ends the transition game at $q_0$, Eloise wins the evaluation game by choosing $\neg\atlf\neg p_1$, as now the value of $\atlf\neg p_1$ is $\open$. 
Suppose that Abelard forces a transition to $q_1$ by choosing $\alpha$ for $a_1$. If he claims $\neg p_1$ is true at $q_1$, Eloise does not challenge. If Abelard ends his seeker turn at $q_1$, Eloise becomes the seeker. At $q_1$ she forces a transition to $q_2$, by choosing $\alpha$ for $a_2$. Then she verifies $\atlf p_2$ by claiming that $p_2$ is true at $q_2$.
If the transition game ends at $q_2$, she wins by choosing $\atlf p_2$, whose value is $\top$.
Note that by following this strategy, Eloise cannot stay as a seeker for infinitely long.
}
\end{example}

We will see later that there is never need for a larger than $|At(\Phi)|$ number of seeker alternations in a transition game for a formula $\coop{A}\Phi$. In Example~\ref{ex:2} we saw that there are cases where exactly $|At(\Phi)|$ seeker alternations are needed in the corresponding transition game. The following example generalizes the setting of Example~\ref{ex:2} by showing that no fixed upper bound for the number of seeker alternations suffices for all transitions games.

\begin{example}

Let $\varphi_k=\coop{a_2}\Psi_k$, where $\Psi_k:=\G r_0\vee\bigvee_{1\leq i\leq k}(\F p_i\wedge\G r_i)$. Consider the following $\CGM$ $\mathcal{M}$ (c.f. the model in Example~\ref{ex:2}).
\begin{center}
\begin{tikzpicture}[
	scale=1.5,
	state/.style={draw, rectangle, rounded corners, font=\small}
	]
	\node at (0,0) [state] (0) {$r_0,\dots, r_n$};
	\node at (1.5,0) [state] (1) {$r_1,\dots, r_n$};
	\node at (3.2,0) [state] (2) {$p_1,\,r_1,\dots, r_n$};
	\node at (4.8,0) [state] (3) {$r_2,\dots, r_n$};
	\node at (6.5,0) [state] (4) {$p_2,\,r_2,\dots, r_n$};
	\node at (5,-1.5) [state] (5) {$p_{n-1},r_{n-1},r_n$};
	\node at (3.4,-1.5) [state] (6) {$r_n$};
	\node at (2.1,-1.5) [state] (7) {$p_n,r_n$};
	\node at (0.7,-1.5) [state] (8) {\phantom{$p_n$}};
	\node[below=4pt of 0] {$q_0$};
	\node[below=4pt of 1] {$q_1$};
	\node[below=0pt of 2] {$q_1'$};
	\node[below=4pt of 3] {$q_2$};
	\node[below=0pt of 4] {$q_2'$};
	\node[below=0pt of 5] {$q_{n-1}'$};
	\node[below=4pt of 6] {$q_n$};
	\node[below=0pt of 7] {$q_n'$};
	\node[below=4pt of 8] {$q_{fin}$};
	\draw[-latex, loop above] (0) to node {\footnotesize $\beta\alpha$} (0);	
	\draw[-latex] (0) to node[below] {\footnotesize $\alpha\alpha$} (1);
	\draw[-latex, loop above] (1) to node {\footnotesize $\alpha\beta$} (1);	
	\draw[-latex] (1) to node[below] {\footnotesize $\alpha\alpha$} (2);
	\draw[-latex, loop above] (2) to node {\footnotesize $\beta\alpha$} (2);
	\draw[-latex] (2) to node[below] {\footnotesize $\alpha\alpha$} (3);
	\draw[-latex, loop above] (3) to node {\footnotesize $\alpha\beta$} (3);
	\draw[-latex] (3) to node[below] {\footnotesize $\alpha\alpha$} (4);
	\draw[-latex, loop above] (4) to node {\footnotesize $\beta\alpha$} (4);
	\draw[-latex, dashed] (4) to (5);
	\draw[-latex, loop above] (5) to node {\footnotesize $\beta\alpha$} (5);
	\draw[-latex] (5) to node[below] {\footnotesize $\alpha\alpha$} (6);
	\draw[-latex, loop above] (6) to node {\footnotesize $\alpha\beta$} (6);
	\draw[-latex] (6) to node[below] {\footnotesize $\alpha\alpha$} (7);
	\draw[-latex, loop above] (7) to node {\footnotesize $\beta\alpha$} (7);
	\draw[-latex] (7) to node[below] {\footnotesize $\alpha\alpha$} (8);
	\draw[-latex, loop left] (8) to node {\footnotesize $\alpha\alpha$} (8);
\end{tikzpicture}
\end{center}

At $q_0$ Eloise wants to end her seeker turn immediately as $\G r_0$ ``still'' true. When Abelard becomes the seeker, he wants to make a transition to $q_1$ and falsify $\G r_0$ there. Since Abelard has then no reason to continue as a seeker, he gives the seeker turn to Eloise. Now Eloise wants to make a transition to $q_1'$ in order to verify $\F p_1$; since $\G r_1$ is still true, Eloise has then no reason to continue as a seeker. We may suppose that the transition game continues like this, so that the seeker role is swapped \emph{after every transition} and $\F p_i$ are verified while $\G r_i$ are falsified. When Abelard finally becomes the seeker at $q_n'$, the maximum number of $|At(\Psi_k)|=2k+1$ seeker alternations has been used. Then Abelard makes a transition to $q_n'$, falsifies $\G r_n$ and wins the ``boolean game'' for $\Psi_k$ with the values of the (fully updated) truth function.
\end{example}


\section{Analysing evaluation games}
\label{sec:results}

In this section we will analyse the properties of the evaluation games of \ATLplus. We first prove positional determinacy of both 
bounded and unbounded evaluation games. Then we find so-called stable timer bounds for bounded evaluation games and show that with them, the bounded \GTS becomes equivalent to the unbounded \GTS. Finally we present the notion of a regular strategy which will be needed for proving the equivalence of \GTS and the standard compositional semantics of \ATLplus in the next section.


\subsection{Positional determinacy}

Here we prove positional determinacy of both bounded and unbounded evaluation games. Recall here that positions are either locations in evaluation games or configurations in transition games---in the extended sense which was discussed in Remark~\ref{remark}.

\begin{proposition}\label{the: bounded determinacy}
Bounded evaluation games 
are determined and the winner has a positional winning strategy.\hspace{-4mm}
\end{proposition}

\begin{proof}
(Sketch) 
Since ordinals are well-founded and they must decrease during transition games, it is easy to see that the game tree is well-founded. Thus positional determinacy
follows easlily, essentially by backward induction.
\end{proof}


\begin{proposition}\label{the: unbounded determinacy}
Unbounded evaluation games 
are determined and the winner has a positional winning strategy.
\end{proposition}

\aamasp{\vspace{-3mm}}

\aamasp{
\begin{proof}
(Sketch)
This claim can be proved 
in a similar way as Gale-Stewart theorem.
%
Another way to prove the claim is to
show that unbounded evaluation games are essentially B\"{u}chi-games
(see, e.g., \cite{Mazala01} for B\"{u}chi-games). 

The details of the proof via B\"{u}chi-games are in \cite{GTS-ATLplus2017-techrep},

but the principal idea is to set up a B\"{u}chi condition such that Eloise wins the B\"{u}chi game if the set of positions visited infinitely often is included in the union of configurations of the transition games where Abelard is the seeker and positions of the evaluation game where Eloise has already won.
\end{proof}
} 
\techrep{ 
\begin{proof}
We will show that unbounded evaluation games are essentially B\"{u}chi-games
(see, e.g., \cite{Mazala01}). 
We first discuss the case where the underlying \CGM $\mathcal{M}$ is finite.
We follow the technicalities for B\"{u}chi-games from \cite{CHP06}, which gives an
excellently detailed and to-the-point presentation of the related basic notions.
Take a triple $(\mathcal{M},q,\varphi)$,
where $\mathcal{M}$ is a finite \CGM, $q$ a state of $\mathcal{M}$, 
and $\varphi$ a formula of $\ATL$.
We will convert this triple into a B\"{u}chi game BG such 
that $\mathcal{M},q\models\varphi$ iff player $2$
has a winning strategy in BG from a certain
position of BG determined by the state $q$.
The required B\"{u}chi game  BG corresponds almost exactly to
the unbounded evaluation game $\mathcal{G}(\mathcal{M},q,\varphi)$.
The set of states of BG is the finite set of positions in  
$\mathcal{G}(\mathcal{M},q,\varphi)$.
The states of BG assigned to player 1 (resp., player 2) of BG are
the positions where Abelard (resp., Eloise) is to move.
The edges of the binary transition relation $E$ of BG
correspond to the  changes of positions in $\mathcal{G}(\mathcal{M},q,\varphi)$.
%
%
%
%
Also, $E$ is defined such that ending locations in the
evaluation game connect (only) to themselves via $E$.
This ensures that every state of BG has a successor state.
We set a \emph{co-B\"{u}chi-objective}
such that an infinite play of BG is
winning for player 2 iff the set of states visited 
infinitely often is a subset of the union of the following sets of states of BG:
\begin{enumerate}
\item
States of BG corresponding to
configurations of the transition games where Abelard is the seeker.
\item
States  of BG corresponding to such ending locations in 
the game $\evalgame(\mathcal{M},q,\varphi)$
where Eloise has already won.  
\end{enumerate}
Clearly, Eloise (resp., Abelard) has a 
positional winning strategy in the evaluation game
starting at a position $\mathit{pos}$ of 
the evaluation game iff player 2  (resp., player 1) in BG has a
positional winning strategy from the state of BG
corresponding to $\mathit{pos}$.
Finite B\"{u}chi games enjoy positional 
determinacy (see e.g. 
\cite{CHP06}), which completes the case of finite $\CGM{s}$. 
For infinite $\CGM{s}$, the argument is the same but requires positional determinacy of B\"{u}chi games on infinite game graphs. That fact is well-known and follows easily from Theorem 4.3 of \cite{grade}.
\end{proof}
} 

By the positional determinacy, we have the following consequence: If Eloise (Abelard) has a perfect recall strategy in a bounded or unbounded evaluation game (or transition game), then she (he) has a positional winning strategy in that game.
%


\subsection{Finding stable timer bounds}

In this section study which timer bounds are ``stable'' for a given model. Intuitively this means that a timer bound $\tlimbound$ is stable for a model $\mathcal{M}$ if neither of the players can benefit from announcing timers that are higher than (or equal to) $\tlimbound$. We will see that, by finding stable timer bounds, we can make the bounded \GTS equivalent to the unbounded \GTS. Moreover, the identification of stable timer bounds for finite models will be necessary for our model checking proofs in Section~\ref{sec:MC}. 

We next consider a ``semi-bounded'' variant of the transition game in which one player must use timers when being the seeker and the other is allowed to play without timers.
A timer bound $\tlimbound$ is \defstyle{stable} for an unbounded transition game $\defstyle{g}(\vrf,q_0,\coop{A}\Phi)$ if the player with a winning strategy in $\defstyle{g}(\vrf,q_0,\coop{A}\Phi)$ can, in fact, win using timers below $\tlimbound$.

\techrep{We first identify stable timer bounds for \emph{finite} models.}
\begin{proposition}
\label{prop:limit}
Let $\mathcal{M}$ be a finite \CGM, $q_0\in\St$ a state and $\Phi\in\ATLplus$ a
path formula. Then $k:=|\St|\cdot|\mathit{At}(\Phi)|$ is a
stable timer bound for $\defstyle{g}(\vrf,q_0,\coop{A}\Phi)$.
\end{proposition}

\begin{proof}
We give a detailed sketch of proof.
Let $c=({\bf E},q,T,n,x)$ be a configuration
(for an \emph{unbounded} game, so
no timer is listed). Suppose that exit location $(\vrf,q,\Phi,T)$ is not a
winning location for Eloise. Then she wants to stay as the
seeker until the truth function is modified to $T'$ that makes $\Phi$ true.
Since $T$ is updated state-wise, it is not beneficial for Eloise to go in loops
such that $T$ is not updated. Hence, if Eloise has a winning strategy from $c$, then she has a winning strategy in which $T$ is updated at least once every $|\St|$ rounds. Since $T$ can be updated at most $|\mathit{At}(\Phi)|$ times, we see that a timer greater than $k =|\St|\cdot|\mathit{At}(\Phi)|$ is not needed.
\end{proof}

\aamasp{
\vspace{-4mm}
}

\begin{corollary}
If $\mathcal{M}$ is a finite \CGM, the unbounded \GTS is equivalent on $\mathcal{M}$ to the $(|\St|\cdot|\varphi|)$-bounded \GTS.
\end{corollary}
\aamasp{\vspace{-0,1cm}}
In order to find stable timer bounds for infinite models, we give the following definition (cf. Def~4.12 in \cite{GKR-AAMAS2016}).
\aamasp{\vspace{-0,1cm}}
\begin{definition}\label{def: branching}
Let $\mathcal{M}$ be a \CGM and let $q\in\St$. The \defstyle{branching degree of $q$}, $\bd(q)$, is the cardinality of the set of outcome states from $q$:
$\bd(q):=\card(\{o(q,\vec\alpha)\mid \vec\alpha\in\avact(\Agt,q)\})$.
The \defstyle{regular branching bound of $\mathcal{M}$},
or $\rbb(\mathcal{M})$, is the smallest infinite regular cardinal $\kappa$
such that $\kappa > \bd(q)$ for every $q\in\St$.
Note that $\rbb(\mathcal{M})=\omega$ if and only if $\mathcal{M}$ is \defstyle{image-finite}.
\end{definition}

If $c=(\ctr,q,T,n,x)$ is a configuration in an unbounded transition game and $\gamma$ is an ordinal,
we use the notation $c[\gamma]:=(\ctr,q,T,n,\gamma,x)$.

\begin{proposition}
\label{prop:RBB} 
Let $\mathcal{M}$ be a \CGM, $q_0\in\St$ and $\Phi\in\ATLplus$ a path formula. Then $\rbb(\mathcal{M})$ is a stable timer bound for $\defstyle{g}(\vrf,q_0,\coop{A}\Phi)$. 
\end{proposition}

\begin{proof}
Suppose first that Eloise has a winning strategy $\tau$ in $\defstyle{g}(\mathcal{M},q_0,\coop{A}\Phi)$. Let $c$ be any configuration of the form $c=(\pl,{\bf A},q,T,n,\defstyle{ii})$ such that
\begin{itemize}
\item $c$ can be reached with $\tau$.
\item If Abelard decides to quit seeking at $c$, then $\tau$ instructs Eloise to become seeker.
\end{itemize}
We need to find an ordinal $\gamma_0<\rbb(\mathcal{M})$ for Eloise to announce if she needs to become seeker at $c$ and supplement $\tau$ with instructions on lowering the ordinal after every transition while she is a seeker. We will use the instructions given by $\tau$ for verifications and choices for actions.

Suppose that Abelard quits seeking at $c$. Let $T_{\defstyle{g},c}$ be the tree that is formed by all of those paths of confiqurations, starting from $c$, in which Eloise stays as the seeker and plays according to $\tau$. Since $\tau$ is a winning strategy, every path in $T_{\defstyle{g},c}$ must be finite, and thus $T_{\defstyle{g},c}$ is well-founded. We prove the following claim by well-founded induction on $T_{\defstyle{g},c}$:
\begin{align*}
	\text{For every } c'\in T_{\defstyle{g},c}, \text{ there is an ordinal } \gamma<\rbb(\mathcal{M}) \\
	\text{ s.t. } c'[\gamma] \text{ is a winning position for Eloise}. 
\end{align*}
We choose $\gamma=0$ for every leaf on $T_{\defstyle{g},c}$. Suppose then that $c'$ is not a leaf. By the inductive hypothesis, the claim holds for every configuration that can be reach with a transition from $c'$. We now define $\gamma$ to be the \emph{successor of the supremum} of these ordinals. Since $\rbb(\mathcal{M})$ is regular, we have $\gamma<\rbb(\mathcal{M})$.
Then, there is $\gamma_0<\rbb(\mathcal{M})$ such that $c[\gamma_0]$ is a winning configuration for Eloise. 
\end{proof}

By using Proposition~\ref{prop:RBB}, it is now easy to show that when the regular branching bound of the given model is used as a timer bound $\tlimbound$, then the $\tlimbound$-bounded \GTS becomes equivalent to the unbounded \GTS. 

\begin{corollary}\label{the: bounded vs. unbounded}
Suppose that $\tlimbound\geq\rbb(\mathcal{M})$. Then the unbounded \GTS is equivalent on $\mathcal{M}$ to the $\tlimbound$-bounded \GTS.
\end{corollary}

\aamasp{\vspace{-2mm}}

\begin{proof}

Suppose first that $\mathcal{M},q\Vdash\varphi$. By Proposition \ref{prop:RBB} Eloise can win the evaluation game using timers smaller than $\tlimbound$ when being the seeker. Hence clearly $\mathcal{M},q\Vdash_\tlimbound\varphi$. 

Suppose then $\mathcal{M},q\not\Vdash\varphi$. By Proposition~\ref{the: unbounded determinacy}, Abelard has a winning strategy in $\evalgame(\mathcal{M},q,\varphi)$. Thus, by Proposition~\ref{prop:RBB}, Abelard can win $\evalgame(\mathcal{M},q,\varphi)$ using timers smaller than $\tlimbound$ when being the seeker. Hence, Abelard clearly has a winning strategy in $\evalgame(\mathcal{M},q,\varphi,\tlimbound)$ and thus $\mathcal{M},q\not\Vdash_\tlimbound\varphi$. 
\end{proof}

Consequently, finite timers suffice in image-finite models. However, the finitely bounded \GTS (with $\tlimbound=\omega$) is not generally equivalent to the unbounded \GTS. See the following example.

\begin{example}[C.f. Example 3.7 in \cite{GKR-AAMAS2016}]\label{ex: omega branching}
Consider the image infinite concurrent game model $\mathcal{M}$ which is displayed in the figure below.
%

\begin{center}
\begin{tikzpicture}
	[
	scale=0.55,
	prestate/.style={rectangle,draw=black!100,fill=black!20,thick,inner sep=3pt,minimum width=5mm,font=\small},
	state/.style={rectangle,rounded corners,draw=black!50,thick,inner sep=3pt,minimum width=5mm,font=\small},
	tb/.style={dashed,-latex},
	sb/.style={-latex},
	action/.style={font=\tiny},
	label/.style={font=\scriptsize}
	]
	\node at (0,3) [state] (s0) {$\neg p$};
		\node [label,above= 1pt of s0] {$s_0$};
	\node at (-6,0) [state] (s10) {$p$};
		\node [label,below= 1pt of s10] {$t_0$};
	\node at (-3,0) [state] (s20) {$\neg p$};
		\node [label,below= 1pt of s20] {$t_1$};
	\node at (0,0) [state] (s30) {$\neg p$};
		\node [label,below= 1pt of s30] {$t_2$};
	\node at (3,0) [state] (s40) {$\neg p$};
		\node [label,below= 1pt of s40] {$t_3$};
	\node at (6,0) [state] (s50) {$\neg p$};
		\node [label,below= 1pt of s50] {$t_4$};
	\node at (7,1.5) {$\cdots$};
	\draw [sb] (s0) to node [action,left] {$0,1$} (s10);
	\draw [sb] (s0) to node [action,right] {$0,2$} (s20);
	\draw [sb] (s0) to node [action,right] {$0,3$} (s30);
	\draw [sb] (s0) to node [action,right] {$0,4$} (s40);
	\draw [sb] (s0) to node [action,right] {$0,5$} (s50);
	\draw [sb] (s20) to node [action,below] {$0,0$} (s10);
	\draw [sb] (s30) to node [action,below] {$0,0$} (s20);
	\draw [sb] (s40) to node [action,below] {$0,0$} (s30);
	\draw [sb] (s50) to node [action,below] {$0,0$} (s40);
	\draw [sb] (8.5,0) to node [action,below] {$0,0$} (s50);
	\draw [sb,loop left] (s10) to node [action,left] {$0,0$} (s10);
\end{tikzpicture}
\end{center}

Here we clearly have $\mathcal{M},s_0\Vdash \coop{1}\F p$ since every path from $s_0$ will eventually reach the state $t_0$ where $p$ is true. However, $\mathcal{M},s_0\not\Vdash_\omega \coop{1}\F p$ since for any value $n<\omega$ for the timer, chosen by Eloise, Abelard can choose $n$ for the first action of agent $2$ and then it will take $n+1$ rounds to reach $t_0$.

Because $\rbb(\mathcal{M})=\aleph_1$ (equal to $2^{\aleph}_0$ if we assume the continuum hypothesis), by Corollary~\ref{the: bounded vs. unbounded} we have $\mathcal{M},s_0\Vdash_{\aleph_1} \coop{1}\F p$. However, in this particular model, we also have $\mathcal{M},s_0\Vdash_{\omega+1} \coop{1}\F p$ since Eloise can win the game by first choosing $\omega$ for the value of the timer and then lowering its value to $n<\omega$ which corresponds the the action which Abelard first chooses for the agent 2. 

\end{example}


\subsection{Regular strategies}

Here we define a notion of a regular strategy which will be important for the proofs later in this paper. We only define this concept for Eloise only for the transition games in which Eloise is the verifier. This suffices for our needs, but the definition---and the related Lemma~\ref{the: regular strategies}---could easily be generalized for both players and all kinds of transition games.

\begin{definition}
A strategy $\tau$ for Eloise in a transition game $\defstyle{g}({\bf E},q,\coop{A}\Phi)$ is \defstyle{regular}, if the following properties hold:
\begin{enumerate}[leftmargin=20pt]
\item[(i)] $\tau$ instructs Eloise to make all the claims which are valid (by the respective \GTS). Moreover, $\tau$~instructs Eloise to challenge all the claims which Abelard makes. (Note that this latter condition is safe for Eloise since she is given the chance to make every claim first and thus, by the first condition, \emph{Abelard can only make claims which are false}.)

\item[(ii)] $\tau$ instructs Eloise to try to end the game (by ending her seeker turn or by not taking a new seeker turn) always when the truth function $T$ has winning values for Eloise---that is, she would a have a winning strategy from the exit location if Abelard did not want to continue as a seeker. 

\item[(iii)] Actions chosen by $\tau$ (for the agents in $A$) are independent of the current seeker $\ctr$ and seeker turn counter $n\in\mathbb{N}$ in configurations.
\end{enumerate}
\end{definition}
Note that the conditions (i)-(iii) together imply that \emph{all} the actions chosen by a regular strategy are independent of the current seeker $\ctr$ and seeker turn counter $n\in\mathbb{N}$ in configurations. Hence, the actions chosen by a regular strategy depend only\footnote{The parameter $x$ and all the other information that is should be encoded in the configurations (see Remark~\ref{remark}) are only used for describing the current sub-phase of the game. Hence, it is easy to see players' strategies cannot depend on these parameters.} on the pairs $(q,T)$, where $q$ is the current state and $T$ is the current truth function. 
Also note that since, by (i), Eloise makes all the valid verifications and falsifications, the truth function $T$ is always determined by the path that has been formed by the transition game.

The following lemma shows that from now on we may assume all winning strategies to be regular. Since regular strategies depend only on the states and the truth function, the additional parameters $\ctr$ and $n$ cannot be used for ``signalling'' any information for $\tau$.

\begin{lemma}\label{the: regular strategies}
If Eloise has a winning strategy in a transition game $\defstyle{g}({\bf E},q,\coop{A}\Phi)$, then she has a regular winning strategy in that game. 
\end{lemma}

\begin{proof}
Suppose that Eloise has winning strategy $\tau$  in $\defstyle{g}({\bf E},q,\coop{A}\Phi)$. 
We first note that, for checking the regularity conditions (i)--(iii), it suffices the we only consider the configurations that can be reached with the strategy of Eloise. This is because we can choose arbitrary actions for all the other configurations in order to satisfy the regularity conditions. 
We make the strategy $\tau$ regular by doing the following modifications (in the given order).
\begin{enumerate}[leftmargin=20pt]
\item If $\tau$ does not satisfy the regularity propety (i), then we simply first modify it so that Eloise makes all the claims which are true by \GTS; it is clear that we end up in Eloise's winning exit location if Abelard challenges these new claims. 
Moreover, we then redefine $\tau$ to challenge all the claims made by Abelard; since all of these claims must now by false by \GTS, it follows from the determinacy of evaluation games that every challenge by Eloise leads into an exit location which is winning for her.
After these modifications, $\tau$ is still a winning strategy and it now satisfies the regularity property (i).

\item Let $c=(\pl,q,T,n,\defstyle{ii})$ be a configuration that can be reached with $\tau$ so  that $({\bf E},q,\Phi,T)$ is a winning location for Eloise, but $\tau$ does not instruct Eloise to try to end the transition game at $c$. We then redefine $\tau$ to instruct Eloise to try to end the game at $c$. If Abelard also wants to end the game, then we reach a winning exit location for Eloise. If Abelard does not want to end the game, then the game continues from a configuration $c'$ that must be winning for Eloise. We can then modify $\tau$ in such way that it is a winning strategy from $c'$. Moreover, we can do this while maintaining the regularity conditions (i) and (ii)---we simply do the same modifications as above for all new configurations that violate these regularity conditions.

After doing the the procedure above for all configurations for which $\tau$ violates the regularity property (ii), $\tau$ satisfies the properties (i) and (ii).

\item In order to satisfy the regularity condition (iii), will first modify $\tau$ in various ways and then show that the modified strategy satisfies the condition (iii). Supposing that $\tau$ already satisfied the conditions (i) and (ii), it will then be regular. 

Suppose first that $c=({\bf A},q,T,n,\defstyle{iii})$ is a winning configuration for Eloise, but $T$ is not winning for Eloise (in the boolean game that potentially follows). Let $c'=({\bf E},q,T,n-1,\defstyle{iii})$. Since Abelard could have ended his Seeker turn at $({\bf A},q,T,n,\defstyle{ii})$, it now follows that $c'$ must be a winning configuration for Eloise. We then modify $\tau$ in such way that it makes the same choice at $c'$ and $c$ (we can do that while maintaining the regularity conditions (i) and (ii) by doing the modifications above---if necessary). We do these modifications for all configurations $c$ of this type.

We then do the following procedure for every integer $n\leq|\mathit{At}(\Phi)|$, beginning from $n=|\mathit{At}(\Phi)|$. Let $c_n=(\pl,q,T,n,\defstyle{iii})$ be a configuration that can be reached with $\tau$. Let $n'\leq|\mathit{At}(\Phi)|$ be the largest integer such that $c_{n'}=(\pl,q,T,n',\defstyle{iii})$ can be reached with $\tau$. We redefine $\tau$ at $c_n$ in such a way that it selects the same actions as at $c_{n'}$. We continue this modification in such a way that, when playing from $c_n$, we can only reach configurations of the same form as those that can be reached from $c_{n'}$, the only difference being the value of seeker alternation counter. Now all the exit locations that can be reached by using $\tau$ from $c_n$ must be winning for Eloise. Since the truth function can be updated at most $|\mathit{At}(\Phi)|$ many times and, by condition (ii), $T$ gets updated after every seeker alternation, it is impossible that Eloise would now lose the game because the seeker turn counter would become zero. Hence $\tau$ is still a winning strategy after these modifications.

We observe that by doing the procedure above for every $n\leq|\mathit{At}(\Phi)|$ (starting from the highest values) and for every configuration $c_n$, we finally obtain a winning strategy that is completely independent of the seeker turn counter. Also note that, by applying this procedure, we also maintain the regularity conditions (i) and (ii) for $\tau$.

To prove that the actions chosen by $\tau$ for $A$ are now independent of \emph{both} the seeker $\ctr$ and the seeker turn counter $n$, suppose for the sake of contradiction that $\tau$ assigns different actions  for $A$ in configurations $c=(\pl,q,T,n,\defstyle{iii})$ and $c'=(\pl',q,T,n',\defstyle{iii})$ such that $c\neq c'$ and both $c$ and $c'$ can be reached with $\tau$. Since $\tau$ is independent of the seeker turn counter, we must have $\pl\neq\pl'$. By symmetry we may assume that $\pl={\bf E}$ and $\pl'={\bf A}$.

Suppose first that $({\bf E},q,\Phi,T)$ is a winning position for Eloise. Now, by the condition (ii), $\tau$ instructs Eloise to end her seeker turn at $({\bf E},q,T,n,\defstyle{ii})$, and thus the configuration $c$ cannot be reached with $\tau$.
Suppose then that $({\bf E},q,\Phi,T)$ is not a winning position for Eloise. Recall that we have defined $\tau$ to make the same choice at $c'$ as at the configuration $c''=({\bf E},q,T,n'-1,\defstyle{iii})$. But this is impossible since $\tau$ is independent of the seeker turn counter and that is the only parameter that separates the configurations $c$ and $c''$.
\end{enumerate}

By doing all the modifications above, $\tau$ becomes a regular strategy. Since it remains a winning strategy for Eloise even after all  these modifications, Eloise thus has a regular winning strategy in $\defstyle{g}({\bf E},q,\coop{A}\Phi)$.
\end{proof}

Regular strategies will play an important role in the next section where we prove the equivalence of \GTS and the standard compositional semantics for \ATLplus. This is because regular strategy of Eloise in a transition game for $\coop{A}\Phi$ can be used in a straightforward way for formulating a collective strategy $S_A$ for the coalition $A$ (and vice versa).


\section{GTS vs compositional semantics for ATL$^{+}$}
\label{sec:equivalence}

In this section we show that our game-theoretic semantics is equivalent to the standard (perfect-recall) compositional semantics of \ATLplus. From the results of the previous section it follows that this equivalence holds for both unbounded \GTS and bounded \GTS with a stable timer bound.

We begin with some preliminary definitions. We first define a so-called \defstyle{finite path semantics},
to be used later. See \cite{BJ} for a similar definition.
We define the \defstyle{length} $\lgt(\lambda)$ of a finite path $\lambda$ as the number of transitions in $\lambda$ (whence the last state of $\lambda$ is $\lambda[\lgt(\lambda)]$).
If $\lambda$ is a prefix sequence of $\lambda'$, we write $\lambda\preceq\lambda'$.
%


%
\aamasp{
\noindent
\begin{definition}\label{def: finite path semantics}
Let $\mathcal{M}$ be a $\CGM$ and $\lambda\in\paths_\text{fin}(\mathcal{M})$.
\defstyle{Truth} of an $\ATL^+$ path
formula $\Phi$ on the finite path $\lambda$ is defined as expected,
the non-obvious clauses being as follows:
\begin{itemize}
\item $\mathcal{M},\lambda\models
\X\varphi$ iff\; $\lgt(\lambda)\geq 1$ and $\mathcal{M},\lambda[1]\models\varphi$.
\item $\mathcal{M},\lambda\models\varphi\U\psi$ iff
there exists some $i\leq \lgt(\lambda)$ such that $\mathcal{M},\lambda[i]\models\psi$ and $\mathcal{M},\lambda[j]\models\varphi$ \,for all $j < i$.
\end{itemize}
\end{definition}
}
\begin{definition}\label{def: finite path semantics}
Let $\mathcal{M}$ be a $\CGM$ and $\lambda\in\paths_\text{fin}(\mathcal{M})$.
\defstyle{Truth} of a path formula $\Phi$ of $\ATL^+$ on $\lambda$ is defined as follows:
%
\begin{itemize}
\item $\mathcal{M},\lambda\models\varphi$
iff $\mathcal{M},\lambda[0]\models\varphi$\ (where $\varphi$ is a 
state formula). 
\item $\mathcal{M},\lambda\models
\X\varphi$ iff\; $\lgt(\lambda)\geq 1$ and $\mathcal{M},\lambda[1]\models\varphi$.
\item $\mathcal{M},\lambda\models
\neg\Phi$ iff $\mathcal{M},\lambda\not\models\Phi$.
\item $\mathcal{M},\lambda\models\Phi\vee\Psi$ iff $\mathcal{M},\lambda\models\Phi$ or $\mathcal{M},\lambda\models\Psi$.
\item $\mathcal{M},\lambda\models\varphi\U\psi$ iff
there exists some $i\leq \lgt(\lambda)$ such that $\mathcal{M},\lambda[i]\models\psi$ and $\mathcal{M},\lambda[j]\models\varphi$ \,for all $j < i$.
\end{itemize}
\end{definition}

\begin{definition}\label{def: truth alternation}
Let $\mathcal{M}$ be a \CGM, $\Lambda\in\paths(\mathcal{M})$ and $\Phi$ a path formula of \ATLplus. 
An index $i\geq 1$ is a \defstyle{truth-swap point} of $\Phi$ on $\Lambda$ if either of the following holds:
\begin{enumerate}
\item $\mathcal{M},\Lambda[i\!-\!1,\infty)\not\models\Phi$ and $\mathcal{M},\Lambda[i,\infty)\models\Phi$.
\item $\mathcal{M},\Lambda[i\!-\!1,\infty)\models\Phi$ and $\mathcal{M},\Lambda[i,\infty)\not\models\Phi$.
\end{enumerate}
(Above the notation $\Lambda[i,\infty)$ denotes the infinite path $(\Lambda[i],\Lambda[i+1],\dots)$.)

We define the \defstyle{truth-swap number} of $\Phi$ on $\Lambda$ to be
\[
	\mathit{TSN}(\Phi,\Lambda) := \card(\{i\mid i \text{ is a truth-swap point of $\Phi$ on $\Lambda$}\}).
\]
\end{definition} 

\noindent
The claims of the following lemma are
easy to prove. Similar observations have been made in \cite{BJ}.

\aamasp{\vspace{-1mm}}

\begin{lemma}\label{the: finite path lemma}
Let $\mathcal{M}$ be a \CGM, $\Lambda\in\paths(\mathcal{M})$ and $\Phi$ a path formula of \ATLplus.
Now, the following claims hold:
\begin{enumerate}
\item $\mathit{TSN}(\Phi,\Lambda)\leq |\{\Psi\in\mathit{At}(\Phi)\, |\, \Psi\text{ is a
temporal subformula}\}|$.
\item $\mathcal{M},\Lambda\models\Phi$ iff there is some $k\in\mathbb{N}$ s.t. $\mathcal{M},\lambda\models\Phi$ for every finite $\lambda\preceq\Lambda$ for which $\lgt(\lambda)\geq k$.
\end{enumerate}
\end{lemma}

\begin{theorem}\label{the: equivalence to compositional semantics}
The unbounded \GTS is equivalent to the standard (perfect-recall) compositional semantics of \ATLplus.
\end{theorem}

\begin{proof}
We prove by induction on $\ATLplus$ state formulae $\varphi$ that for any 
CGM $\mathcal{M}$ and a state $q$ in $\mathcal{M}$: 
\[
	\mathcal{M},q\models\varphi \; \text{ iff } \; \text{Eloise has a winning strategy in } \mathcal{G}(\mathcal{M},q,\varphi).
\]
If $\varphi$ is a proposition symbol, then the claim holds trivially.

Let $\varphi=\neg\psi$ and suppose first that $\mathcal{M},q\models\neg\psi$, i.e. $\mathcal{M},q\not\models\psi$. By the inductive hypothesis Eloise does not have a winning strategy in $\mathcal{G}(\mathcal{M},q,\psi)$. Since evaluation games are determined, Abelard has a winning strategy in $\mathcal{G}(\mathcal{M},q,\psi)$. Thus, Eloise has a winning strategy in $\mathcal{G}(\mathcal{M},q,\neg\psi)$.
Suppose then that Eloise has a winning strategy in the evaluation game $\mathcal{G}(\mathcal{M},q,\neg\psi)$. Then Eloise cannot have a winning strategy in $\mathcal{G}(\mathcal{M},q,\psi)$. Hence, by the inductive hypothesis, $\mathcal{M},q\not\models\psi$, i.e. $\mathcal{M},q\models\neg\psi$.

Let $\varphi=\psi\vee\theta$ and suppose that $\mathcal{M},q\models\psi\vee\theta$, i.e. $\mathcal{M},q\models\psi$ or $\mathcal{M},q\models\theta$. Suppose first that $\mathcal{M},q\models\psi$, whence by the inductive hypothesis Eloise has a winning strategy in $\mathcal{G}(\mathcal{M},q,\psi)$. Now Eloise can win $\mathcal{G}(\mathcal{M},q,\psi\vee\theta)$ by choosing $\psi$ on the first move. The case when $\mathcal{M},q\models\theta$ is analoguos. 
Suppose now that Eloise has a winning strategy in the evaluation game $\mathcal{G}(\mathcal{M},q,\psi\vee\theta)$. Let $\chi\in\{\psi,\theta\}$ be disjunct that Eloise chooses when following her winning strategy. Now Eloise must have a winning strategy in $\mathcal{G}(\mathcal{M},q,\chi)$ and thus by the inductive hypothesis $\mathcal{M},q\models\chi$. Therefore $\mathcal{M},q\models\psi\vee\theta$.

Finally, let $\varphi=\coop{A}\Phi$. It suffices to show that Eloise has winning strategy in the (unbounded) transition game $\defstyle{g}({\bf E},q,\coop{A}\Phi)$ if and only if the coalition $A$ has a (perfect recall) strategy $S_A$ such that $\mathcal{M},\Lambda\models\Phi$ for every $\Lambda\in\paths(q,S_A)$. The cases (a) and (b) which follow correspond to the two directions of this equivalence.

\medskip

(a) Suppose first that ${\bf E}$ has a winning strategy $\tau$ in the transition game $\defstyle{g}({\bf E},q,\coop{A}\Phi)$. By Lemma~\ref{the: regular strategies} we may assume that $\tau$ is regular. Let $T_\defstyle{g}$ be the game tree that is formed by all of those configurations that can be encountered with $\tau$. We define $S_A$ by using the actions according to $\tau$ for every \emph{finite path of states} that occurs in consecutive configurations in $T_\defstyle{g}$. The actions for all other finite paths are irrelevant.

In order to show that $S_A$ is well-defined this way, let $\lambda,\lambda'$ be finite branches of \emph{configurations} in $T_\defstyle{g}$ such that the states occurring in configurations of $\lambda$ and $\lambda'$ are in the same order. Let $c=(\pl,q,T,n,\defstyle{iii})$ and $c'=(\pl',q,T',n',\defstyle{iii})$ be the last configurations in $\lambda$ and $\lambda'$, respectively. It suffices to show that $\tau$ assigns the same actions for $A$ in both $c$ and $c'$. Since $\lambda$ and $\lambda'$ have visited the same states, by regularity condition (i), we must have $T=T'$. Therefore, by regularity condition (iii), $\tau$ assigns the same actions for $c$ and $c'$.

Let $\Lambda\in\paths(q,S_A)$, whence states in $\Lambda$ occur in some infinite tuple of configurations in $T_\defstyle{g}$.
In the (infinite) play of $\defstyle{g}({\bf E},q,\coop{A}\Phi)$, that corresponds to $\Lambda$, Eloise does only finitely many verifications and cannot stay as a seeker for infinitely many rounds (since $\tau$ is a winning strategy). 
Let $k\in\mathbb{N}$ be such that Eloise neither does any further verifications nor becomes a seeker after the state $\Lambda[k]$.
Let $\lambda_0\preceq\Lambda$ be a finite path such that  $|\lambda_0|\geq k$. 

We can show by induction on the formulae in $\subf_{\mathit{At}}(\Phi)$ that if a position of the form $(\pl,\lambda_0[l],\Psi,T)$, where $\Psi\in \subf_{\mathit{At}}(\Phi)$, can be reached by using $\tau$, then the following holds:
\[
	\mathcal{M},\lambda_0\models\Psi \; \text{ iff } \pl = {\bf E}.
\]
\begin{itemize}
\item The cases $\Psi=\varphi$ and $\Psi=\atlx\varphi$ are easy to prove.

\item Let $\Psi=\psi\atlu\theta$ and suppose first that $\pl={\bf E}$. Since $\tau$ is a regular winning strategy, there must be $i\leq k$ s.t. Eloise verifies $\psi\atlu\theta$ at $\lambda_0[i]$. If Abelard challenged Eloise's claim, the evaluation game would have continued from the position $({\bf E},\lambda_0[i],\theta,T)$. By the (outer) inductive hypothesis we have $\mathcal{M},\Lambda[i]\models\theta$. Let then $j<i$. Now Abelard could have attempted to falsify $\psi$ at $\Lambda[j]$, whence Eloise must have challenged since $\tau$ is a regular winning strategy. Then the evaluation game would have continued from the position $({\bf E},\Lambda[j],\psi,T)$ and thus by the (outer) inductive hypothesis $\mathcal{M},\Lambda[j]\models\psi$. Thus we have shown that $\mathcal{M},\lambda_0\models\psi\atlu\theta$.

Suppose now that $\pl={\bf A}$. We also suppose, for the sake of contradiction, that $\mathcal{M},\lambda_0\models\psi\atlu\theta$. Now there is $i\leq k$ such that  $\mathcal{M},\lambda_0\models\theta$. If Abelard would have verified $\theta$ at $\lambda_0[i]$, then Eloise would have lost by the (outer) inductive hypothesis. Hence Eloise should have falsified $\psi\atlu\theta$ at some state $\lambda_0[j]$, where $j<i$. But then by the (outer) inductive hypothesis we must have $\mathcal{M},\lambda_0[j]\not\models\psi$, which is a contradiction. 

\item Suppose that $\Psi=\neg\Theta$. The next position of the evaluation game is $(\opl,\lambda[l],\Theta,T)$ and thus by the (inner) inductive hypothesis, $\mathcal{M},\lambda_0\not\models\Theta$ iff $\opl={\bf A}$. Hence,  we have $\mathcal{M},\lambda_0\models\neg\Theta$ iff $\pl={\bf E}$

\item The case $\Psi=\Theta_1\vee\Theta_2$ is proven similarly to the previous case.
\end{itemize}

Abelard is the seeker at the last state $\lambda_0[m]$ of $\lambda_0$ and may attempt to end the transition game at $\lambda_0[m]$. By our assumption Eloise does not become a seeker and thus the evaluation game is continued from $({\bf E},\lambda_0[m],\Phi,T)$ for some~$T$. By the induction proof above, we must have $\mathcal{M},\lambda_0\models\Phi$. Hence, by Lemma~\ref{the: finite path lemma} we have $\mathcal{M},\Lambda\models\Phi$.

\medskip

(b) Suppose then that there is a joint (perfect recall) strategy $S_A$ such that  $\mathcal{M},\Lambda\models\Phi$ for every $\Lambda\in\paths(q,S_A)$. We define a \emph{perfect recall} strategy $\tau$ for Eloise as follows. Suppose that game is at some configuration $c$ that is reached with a finite path $\lambda_0$ such that $q_0$ is the last state of $\lambda_0$.
\begin{itemize}
\item If $\mathcal{M},q_0\models\theta$ for some $\psi\atlu\theta\in At(\Phi)$, then Eloise claims that $\theta$ is true.
\item If $\mathcal{M},q_0\not\models\psi$ for some $\psi\atlu\theta\in At(\Phi)$, then Eloise claims that $\psi$ is false.
\item Suppose that $q_0=\Lambda[0]$ and $\psi\in\mathit{At}(\Phi)$ is a state formula. If $\mathcal{M},q_0\models\psi$, then Eloise claims that $\psi$ is true.
\item Suppose that $q_0=\Lambda[1]$ and $\atlx\psi\in\mathit{At}(\Phi)$. If $\mathcal{M},q_0\models\psi$, then Eloise claims that $\atlx\psi$ is true.
\item If Abelard makes any claim on the truth of formulae, Eloise always challenges those claims. (Note here that Abelard's claim must be false---according to the compositional truth condition---otherwise Eloise would already have made the same claim by herself.)
\item If Eloise is the seeker in $c$ and $\mathcal{M},\lambda_0\models\Phi$, then Eloise decides to end her seeker turn. 
\item If Abelard ends the seeking at $c$ and $\mathcal{M},\lambda_0\not\models\Phi$, then Eloise decides to become seeker. Otherwise, Eloise ends the transition game at $c$.
\item If Eloise needs to choose actions for agents in coalition $A$ at $c$, she chooses them according to $S_A(\lambda_0)$.
\end{itemize}

We show by (co)-induction on the configurations of the transition game $\defstyle{g}({\bf E},q,\coop{A}\Phi)$, that when Eloise uses $\tau$ she cannot end up in a losing ending position. 
\begin{itemize}
\item Let $c=({\bf E},\ctr,q',T,n,\defstyle{i})$. Since the verifications and challenges are made according to the compositional semantics on the current state, Eloise has a winning strategy from any possible exit position by the (outer) inductive hypothesis.

\item Let $c=({\bf E},\ctr,q',T,n,\defstyle{ii})$. By Lemma~\ref{the: finite path lemma} and the definition of $\tau$, the transition game can only end when $\mathcal{M},\lambda_0\models\Phi$. Hence from the exit position $({\bf E},q',\Phi,T)$, Eloise can play in such a way that for any position $(\pl,q',\Psi,T)$, that is reached, the following condition holds:
\[
	\mathcal{M},\lambda_0\models\Psi \; \text{ iff } \pl = {\bf E},
\]
where $\Psi$ is a subformula of $\Phi$ such that there is $\varphi\in\mathit{At}(\Phi)$ which is a subformula of $\Psi$. Eventually, a location of the form $(\pl,q',\varphi,T)$ is reached, where $\varphi\in\mathit{At}(\Phi)$. Since the verifications by $\tau$ are made according to the compositional truth of the relational atoms of $\Phi$, it is quite obvious to see that $(\pl,q',\varphi,T)$ is a winning position for Eloise.

\item Let $c=({\bf E},\ctr,q',T,n,\defstyle{iii})$. This configuration does not lead to any exit locations.
\end{itemize}

Since Eloise chooses actions for agents in $A$ according to $S_A$, every path of \emph{states} that is formed with $\tau$ is a prefix sequence of some path $\Lambda\in\paths(q,S_A)$. Since $\mathcal{M},\Lambda\models\Phi$ for every $\Lambda\in\paths(q,S_A)$, by Lemma~\ref{the: finite path lemma}, and the definition of $\tau$, Eloise cannot stay as a seeker forever when playing with $\tau$. If Abelard stays as a seeker forever, then Eloise wins. Hence, $\tau$ is a (perfect recall) winning strategy for Eloise.
Since unbounded transition games are positionally determined, there is also a \emph{positional} winning strategy $\tau'$ for Eloise.
\end{proof}

By combining Theorem \ref{the: equivalence to compositional semantics} and Corollary \ref{the: bounded vs. unbounded}, we immediately obtain the following corollary:

\begin{corollary}
If $\tlimbound\geq\rbb(\mathcal{M})$, then the $\tlimbound$-bounded \GTS is equivalent on $\mathcal{M}$ with the standard (perfect recall) compositional semantics of \ATLplus.
\end{corollary}

\section{Model checking ATL$^{+}$ using GTS}
\label{sec:MC}

\noindent \techrep{Here we apply the \GTS to model checking problems for \ATLplus and its fragments.}

\subsection{Revisiting the $\mathrm{PSPACE}$\, upper bound proof}

As mentioned earlier, the $\mathrm{PSPACE}$ upper bound proof
for the model checking of $\ATL^+$ in \cite{BJ} contains a flaw. Indeed, the claim of Theorem 4 in \cite{BJ} is incorrect and a  counterexample to it can be extracted from our Example \ref{ex:2}, where $\mathcal{M}, q_{0} \models \varphi$ for $\varphi =\coop{a_2}(\atlg p_1 \vee \atlf p_2)$. In the notation of \cite{BJ},
since $|St_{\mathcal{M}} | = 3$ and $\mathcal{APF}(\varphi) = 2$,  by the claim there must be a 6-witness strategy for the agent $2$ for 
$(\mathcal{M}, q_{0}, \atlg p_1 \vee \atlf p_2)$. However, this is not the case, since the player 1 can choose to play at $q_{0}$ four times $\beta$, and then  
$\alpha$. 
Then $\mathcal{M}, \Lambda \not\models^{6} (\atlg p_1 \vee \atlf p_2)$ on any resulting path $\Lambda$. 

The reason for the problem indicated above is that compositional semantics easily ignores the role and power of the falsifier (Abelard) in the formula evaluation process.
Still, using the \GTS introduced above, we will demonstrate in a simple way
that the upper bound result is indeed correct.

The input to the model checking problem of $\ATL^+$ is an $\ATL^+$
formula $\varphi$, a finite \CGM $\mathcal{M}$ and 
a state $q$ in $\mathcal{M}$. 
We assume that $\mathcal{M}$ is 
encoded in the standard way (cf. \cite{AHK02,BJ})
that provides a full explicit description of the transition function $o$. 
Unlike \cite{AHK02,BJ}, we do not assume any bounds on the number of
proposition symbols or agents in the input. 
We only consider here the semantics of $\ATL^+$ based on perfect information and perfect-recall strategies. 

\aamasp{\vspace{-0.1cm}}

\begin{theorem}[\cite{BJ}]\label{pspace}
The $\ATL^+$ model checking problem is $\mathrm{PSPACE}$-complete.
\end{theorem}
\begin{proof}
We get the lower bound directly from \cite{BJ}, so we
only prove the upper bound here.
By Theorem~\ref{the: equivalence to compositional semantics} and Proposition~\ref{prop:limit},
if $\mathcal{M}$ is a finite \CGM, we
have $\mathcal{M},q\models\varphi$ iff
Eloise has a positional winning strategy in $\mathcal{G}(\mathcal{M},q,\varphi,N)$
with $N = |\St|\cdot|\varphi|$.
It is routine to construct an alternating Turing machine TM that 
simulates $\mathcal{G}(\mathcal{M},q,\varphi,N)$ such that the positions for Eloise correspond to existential states of TM and Abelard's positions to universal states. 
Due to the timer bound $N$, the machine runs in polynomial time.
It is clear that if Eloise has a (positional or not) winning strategy in the evaluation game, then TM accepts. 
Conversely, if TM accepts, we can read a non-positional winning strategy for Eloise from the the computation tree (with only one successful move for existential states
recorded everywhere) which demonstrates 
that TM accepts. By Proposition  \ref{the: bounded determinacy},
Eloise thus also has a positional winning strategy in the evaluation game.
%
%
%
%
%
%
Since $\mathrm{APTIME} = \mathrm{PSPACE}$, the claim follows.
\end{proof}

\aamasp{\vspace{-0.2cm}}

\subsection{A hierarchy of tractable fragments of \ATLplus}

We now identify a natural hierarchy of tractable fragments of $\ATL^+$.
Let $k$ be a positive integer. Define $\ATL^k$
to be the fragment of $\ATLplus$ where all formulae $\coop{A}\Phi$ have the property that 
$|\mathit{At}(\Phi)|\leq k$. 
Note that $\ATL^{1}$ is essentially the same as \ATL (with Release).
Note also that the number of non-equivalent formulae of $\ATL^k$ is
not bounded for any $k$, \emph{even in
the special case where the number of propositions and actions is constant},
because nesting of strategic operators $\coop{A}$ is not limited. 
%
Still, we will show that the model checking problem for $\ATL^k$ is $\mathrm{PTIME}$-complete for any fixed $k$.
Again $\CGM$s are encoded explicitly and no
restrictions on the number of propositions or
actions is assumed.
\techrep{(In fact, a certain implicit
encoding of $\CGM$s leads to  $\Delta^{\mathrm{P}}_3$-completeness \cite{laroussinie}.)}
%
%
%

%
With the fully developed \GTS in place, the following
theorem is now actually straightforward to prove. This demonstrates the
potential advantages of \GTS.
%

%

\aamasp{\vspace{-1mm}}
\begin{theorem}\label{the: hierarchy of tractable fragments}
For any fixed $k\in \bbN$, the model checking problem for 
$\ATL^k$  is $\mathrm{PTIME}$-complete. 
\end{theorem}
\begin{proof}
The claim is well-known for $\ATL$ (see \cite{AHK02}), so we
have the lower bound for free, for any $k$. 
One possible proof strategy for the upper bound would involve using
alternating $\mathrm{LOGSPACE}$-machines, but here we argue via B\"{u}chi-games instead.
%
%
\aamasp{
See the details of the reduction of unbounded evaluation
games to B\"{u}chi-games in the technical report \cite{GTS-ATLplus2017-techrep} 
(the proof for Proposition \ref{the: unbounded determinacy}).}

Consider a triple $(\mathcal{M},q,\varphi)$,
where $\varphi\in\ATL^k$. By the proof of 
Proposition \ref{the: unbounded determinacy},  
there exists a B\"{u}chi game BG
such that Eloise wins the
unbounded evaluation game $\mathcal{G}(\mathcal{M},q,\varphi)$ iff
she wins BG from the state of BG that corresponds to the
beginning position of the evaluation game. 
We then observe that
since we are considering $\ATL^k$ 
for a \emph{fixed} $k$, the domain size of each truth function $T$
used in the evaluation game is at most $k$,
and thus the number of positions in
$\mathcal{G}(\mathcal{M},q,\varphi)$ is \emph{polynomial} in the 
size of the input $(\mathcal{M},q,\varphi)$. 
(Cf. Remark~\ref{remark} for all the information that should be
encoded in a position in \emph{bounded} evaluation games;
here we only use the simpler unbounded games.)
Thus also the size of BG is polynomial in the input size. 

We note that, 
in order to avoid blow-ups, it is
essential that the maximum domain size $k$ of
truth functions $T$ is fixed.
We also note---as mentioned already in \cite{AHK02}---that the
number of transitions in $\mathcal{M}$ is not bounded
by the square of the number of states of $\mathcal{M}$.
In fact, because we impose no limit (other than finiteness) on the 
number of actions in $\mathcal{M}$,
the number of transitions in relation to states is arbitrary. However,
this is no problem to us since an explicit 
encoding of $\mathcal{M}$---which lists all
transitions explicitly---is part of the input to the model checking problem.
Since B\"{u}chi games can be solved in $\mathrm{PTIME}$, the claim follows.
\end{proof}

\vcut{
Note that the above proof provides more than just a
reduction of unbounded evaluation games on finite models to
B\"{u}chi games. It shows that such
evaluation games essentially \emph{are} B\"{u}chi games.
}

\aamasp{\vspace{-0.3cm}}

\section{Bounded memory semantics for $\ATL^k$}
\label{sec:BoundedMemory}


Strategies with bounded memory in concurrent game models 
can be naturally defined using \emph{deterministic finite state transducers} (or, Mealy machines). For a transducer-based definition of bounded memory strategies, see e.g. \cite{Vester13}, and see \cite{BrihayeLLM09} for more on this topic.
Using such strategies, an agent's moves are determined both by the current state in the model and by the current state (\emph{memory cell}) of the agent's transducer. Then,  transitions take place both in the model and in the state space of the transducer, thus updating the agent's memory.  
\techrep{
So, such strategies are positional with respect to the product of the two state spaces. 
}
In the compositional \defstyle{$m$-bounded memory semantics} ($\models^m$) for $\ATL^+$, agents are allowed to use at most $m$ memory cells, i.e., strategies defined by transducers with at most  $m$ states.  


\subsection{An upper bound for the number of memory cells}

Since the use of the truth function $T$ in our \GTS is analogous to the use of memory cells in $m$-bounded memory semantics, we obtain the following result.


%
\begin{theorem}\label{the: Upper bound for memory cells 1}
For $\ATL^k$, the unbounded \GTS is equivalent to the $m$-bounded memory semantics for $m = 3^k-2^k$.
\end{theorem}

\begin{proof}
Let $m:=3^k-2^k$ and $\varphi\in\ATL^k$. 
We show that 
\[
	\mathcal{M},q\Vdash\varphi \; \text{ iff } \; \mathcal{M},q\models^m\varphi.
\]
 The implication from right to left is immediate by Theorem~\ref{the: equivalence to compositional semantics}. We prove the other direction by induction on $\varphi$. The only interesting case is when $\varphi=\coop{A}\Phi$. Suppose that Eloise has a winning strategy in $\defstyle{g}({\bf E},q,\coop{A}\Phi)$. By Lemma~\ref{the: regular strategies} we may assume that $\tau$ is regular.

We define a memory transducer $\mathcal{T}$ that Eloise can use to define strategies for all agents in $A$. We fix the set of states $C$ of $\mathcal{T}$ to be the set of all 
truth functions $T$ for $\mathit{At}(\Phi)$ such that $T(\chi)=\open$ for at least one $\chi\in\mathit{At}(\Phi)$. Since $T(\chi)\in\{\open,\top,\bot\}$, we have $|C|\leq 3^k-2^k=m$. The initial state of $\mathcal{T}$ is $T_0$ 
where $T_0(\chi)=\open$ for every $\chi\in\mathit{At}(\Phi)$.
The transitions in $\mathcal{T}$ 
 are defined according to how Eloise updates the truth function $T$ during the transition game. However, when $T$ becomes fully updated (i.e. $T(\chi)\neq\open$ for every $\chi\in\mathit{At}(\Phi)$), then no further transitions are made, because in this case all relative atoms have been verified/falsified and the truth of $\Phi$ on the path is fixed.

Now, the strategy 
for each agent $a\in A$ is defined positionally on $C\times\St$
as follows: At a state $T$ of $\mathcal{T}$ and 
state $q \in \mathcal{M}$, the agent $a$ follows the action prescribed by Eloise's winning strategy for the corresponding step phase in the transition game. 
The strategy for $A$ is now well-defined since $\tau$ is regular and thus depends only on the current state and the current truth function.

It is now easy to show that $\mathcal{M},\Lambda\models^m\Phi$ for any path $\Lambda$ that is consistent with the resulting collective strategy for the coalition $A$. 
\end{proof}

By Theorem~\ref{the: equivalence to compositional semantics}, we obtain the following corollary.
\begin{corollary}
For $\ATL^k$, the perfect recall compositional semantics is equivalent to the $(3^k-2^k)$-bounded memory semantics.
\end{corollary}

This extends the known fact that positional strategies 
(using 1 memory cell) 
suffice for the semantics of \ATL (which is essentially the same as $\ATL^1$).
%
Moreover, given a formula, there is no need for the full perfect recall semantics, as we may equivalently apply the bounded memory semantics with a bound that is based on the structure of the formula (``the maximum temporal width'').

By $\ATL_{\F}^k$ we denote the fragment of $\ATL^k$ where all the relative atoms are of the form $\F\varphi$, that is, the ``temporal objectives'' $\Phi$ are
boolean combinations of reachability objectives.

\begin{theorem}\label{the: Upper bound for memory cells 2}
For $\ATL_{\F}^k$, the unbounded \GTS is equivalent to the $m$-bounded memory semantics for $m = 2^k-1$.
\end{theorem}

\begin{proof}
In $\ATL_{\F}^k$ we may modify the rules of the transition games in such a way that relative atoms cannot be falsified by the players (but naturally they can be verified). This is because $\F\psi$ is interpreted as $\top\U\psi$ and $\top$ is never false: if a player tried to falsify $\top\U\psi$, that player would immediately lose once the other player challenges the claim.
With this modification of the rules, there are at most $2^k$ different truth functions that may
appear in the transition games for $\ATL_{\F}^k$. Moreover, there is only a single truth function that is fully updated. Hence we may define a memory transducer $\mathcal{T}$ with $2^k-1$ states as in the proof of Theorem~\ref{the: Upper bound for memory cells 1} and prove the rest of the claim analogously.
\end{proof}

In the next subsection we will show that the result of Theorem~\ref{the: Upper bound for memory cells 2} is optimal in the sense that no smaller number of memory cells guarantees an equivalent semantics. Hence, even for $\ATL_{\F}^k$, the agents may need exponentially many memory cells with respect to the number of relative atoms. 


\subsection{A lower bound for the number of memory cells}

In this section we will investigate the following simple $\ATL_{\F}^k$-formula:
\[
	\xi_k:=\coop{a_1}\Phi_k, \quad\text{ where } \Phi_k := \F p_1\wedge\cdots\wedge\F p_k.
\]
Note that $\Phi_k$ is just a conjunction of reachability goals that agent $a_1$ needs to fulfill (in any order). Since positional strategies suffice for single reachability objectives, it would be intuitive to think that $a_1$ needs at most $k-1$ memory cells in order to achieve $\Phi_k$. This is because $a_1$ needs to change its positional strategy only when completing some of the reachability objectives.\footnote{This can be seen by analyzing our \GTS for \ATLplus: note that (1) the strategies in transition games may be assumed to be positional with respect to the truth function; and (2) the truth function for $\Phi_k$ can be updated at most most $k$ times during the transition game for $\Phi_k$.} However, we will see that the bounded memory strategy of $a_1$ must potentially use a transducer that has exponentially many states with respect to $k$. The model that we will use for proving this claim is constructed in the following example.

\begin{example}\label{ex: A model for exponential memory}
Let $[k]:=\{1,\dots, k\}$ and $\mathcal{M}_k := (\Agt, \St, \Prop, \Act, d, o, v)$ be a \CGM, where
\begin{itemize}[itemsep=-1mm]
\item $\Agt=\{a_1,a_2\}$, 
$\Prop=\{p_1,\dots,p_k\}$;

\item $\Act=[k]\,\cup\,\{B\mid B\subseteq\mathcal{P}([k])\setminus\{\emptyset\}\}\cup\{\void\}$;

\item $\St=\{q_0\}\cup\{q_i\mid i\in[k]\}\cup\bigl\{q_B\mid B\in\mathcal{P}([k])\setminus\{\emptyset,[k]\}\bigr\}$;

\item $v(p_i)=\{q_i\}\cup\{q_B\in\St \mid i\in B\}$ for all $p_i\in\Prop$;

\item $d(q_0,a_1)=\{B\mid B\in\mathcal{P}([k])\setminus\{\emptyset\}\}$, \; $d(q_0,a_2)=[k]$ \\
and $d(q,a_i)=\{\void\}$ when $q\in\St\setminus\{q_0\}$ and $i\in\{1,2\}$;

\item $o(q_0,(B,i))=\begin{cases}q_i \quad\text{if } i\in B,\\ q_B \quad\text{else;}\end{cases}$ \\[1mm]
$o(q_i,(\void,\void))=q_0$ when $i\in[k]$ \\
and $o(q_B,(\void,\void))=q_B$ when $B\in\mathcal{P}([k])\setminus\{\emptyset,[k]\}$.
\end{itemize}
See the following figure for model $\mathcal{M}_k$ in the special case when $k=3$.

\begin{center}
\begin{tikzpicture}[scale=0.9,
	state/.style={draw,circle,inner sep=2pt,minimum width=8mm,font=\tiny},
	action/.style={-latex,font=\tiny}]
	\node at (-4.2,3.3) {$\mathcal{M}_3:$};
	\node at (0,0) [state] (0) {\phantom{$p_4$}};
	\node at (0,3) [state] (1r) {$p_1$};
	\node at (2.9,-1.5) [state] (2r) {$p_2$};
	\node at (-2.9,-1.5) [state] (3r) {$p_3$};
	\node at (2.2,2.2) [state] (1) {$p_1$};
	\node at (1,-2.8) [state] (2) {$p_2$};
	\node at (-3.2,0.8) [state] (3) {$p_3$};
	\node at (3.2,0.8) [state] (12) {$p_1,p_2$};
	\node at (-2.2,2.2) [state] (13) {$p_1,p_3$};
	\node at (-1,-2.8) [state] (23) {$p_2,p_3$};
	\node[above=0pt of 0] {$q_0$};
	\node[above=0pt of 1r] {$q_1$};
	\node[right=0pt of 2r] {$q_2$};
	\node[left=0pt of 3r] {$q_3$};
	\node[right=0pt of 1] {$q_{\{1\}}$};
	\node[right=0pt of 2] {$q_{\{2\}}$};
	\node[below=0pt of 3] {$q_{\{3\}}$};
	\node[below=0pt of 12] {$q_{\{1,2\}}$};
	\node[left=0pt of 13] {$q_{\{1,3\}}$};
	\node[left=0pt of 23] {$q_{\{2,3\}}$};
	\draw[action, right, bend right, near end, text width=5mm] (0) to node 
		{$\{1,2,3\},1$ $\{1,2\},1$ $\{1,3\},1$ $\{1\},1$} (1r);
	\draw[action, left, bend right, near start, text width=3mm] (1r) to node {$\void$, $\void$} (0);	
	\draw[action, sloped, bend right, text width=7mm] (0) to node 
		{$\{1,2,3\},2$ $\{1,2\},2$ $\{2,3\},2$ $\{2\},2$ $\phantom{9000}$ $\phantom{9000}$} (2r);
	\draw[action, right, bend right, near start] (2r) to node {$\void,\void$} (0);
	\draw[action, sloped, bend right, text width=11mm] (0) to node 
		{$\phantom{90000}$ $\phantom{90000}$ $\{3\},3$ $\{1,3\},3$ $\{2,3\},3$ $\{1,2,3\},3$} (3r);
	\draw[action, below, bend right, near start] (3r) to node {$\void,\void$} (0);
	\draw[action, below, near end, text width=1mm] (0) to node {$\{1\},2$ $\{1\},3$} (1);
	\draw[action, left, near end, text width=7mm] (0) to node {$\{2\},1$ $\{2\},3$} (2);
	\draw[action, above, near end, text width=7mm] (0) to node {$\{3\},1$ $\{3\},2$} (3);
	\draw[action, below, sloped, near end] (0) to node {$\{1,2\},3$} (12);
	\draw[action, above, sloped, near end] (0) to node {$\{2,3\},1$} (23);
	\draw[action, above, sloped, near end] (0) to node {$\{1,3\},2$} (13);
	\draw[action, loop above] (1) to node {$\void,\void$} (1);
	\draw[action, loop below] (2) to node {$\void,\void$} (2);
	\draw[action, loop left, text width=3mm] (3) to node {$\void$, $\void$} (3);
	\draw[action, loop right, text width=3mm] (12) to node {$\void$, $\void$} (12);
	\draw[action, loop above] (13) to node {$\void,\void$} (13);
	\draw[action, loop below] (23) to node {$\void,\void$} (23);
\end{tikzpicture}
\end{center}

The model $\mathcal{M}_k$ can be described as follows: At $q_0$ the agent $a_1$ gets to ``announce'' any nonempty set $B$ of (indices of) proposition symbols in $\Prop$. Then, depending on the action chosen by the agent $a_2$, one of the following happens:
\begin{enumerate}
\item \emph{Some} proposition symbol $p_i$, for which $i\in B$, is reached and then the game returns to $q_0$. This happens when $a_2$ chooses $i\in B$, whence a transition is made to $q_i$ and then back to $q_0$. 

\item \emph{All} proposition symbols $p_i$ with $i\in B$ are reached, but thereafter no new proposition symbols can be reached. This happens when $a_2$ chooses some $i\notin B$, whence a transition is made to $q_B$, where the game will loop forever.
\end{enumerate}

We will show that agent $a_1$ has a $(2^k-1)$-bounded memory strategy $\sigma_{a_1}$ which guarantees the truth of $\Phi_k$ on every path in $\paths(q_0,\sigma_{a_1})$. 
We first define a finite state transducer $\mathcal{T}_k$ as follows:
\begin{itemize}
\item The set of states $C$ of $\mathcal{T}_k$ is $\{c_B\mid B\in\mathcal{P}([k])\setminus\{\emptyset\}\}$. Now $|C|=2^k-1$. 
\item The initial state of $\mathcal{T}_k$ is $c_{[k]}$. 
\item The transitions of $\mathcal{T}_k$ are define as follows: Suppose that the current state of $\mathcal{T}_k$ is $c_B$ for some $B\in\mathcal{P}([k])\setminus\{\emptyset\}$ and a state $q_j$ is reached for some $j\in[k]$. Now if $j\in B$ and $B\neq\{j\}$, then $\mathcal{T}_k$ changes its state to $c_{B\setminus\{i\}}$. Else, no transition is made.
\end{itemize}
See the following picture for the transducer $\mathcal{T}_k$ in the special case when $k=3$. 
\begin{center}
\begin{tikzpicture}[scale=1,
	state/.style={draw,circle,minimum width=7.7mm},
	action/.style={-latex,font=\small}]
	\node at (-3.2,0.5) {$\mathcal{T}_3:$};
	\node at (0,0) [state] (123) {};
		\node at (0,0) [draw,circle,minimum width=6.3mm] {};
	\node at (-2,-1.5) [state] (12) {};
	\node at (0,-1.5) [state] (13) {};
	\node at (2,-1.5) [state] (23) {};
	\node at (-2,-3) [state] (1) {};
	\node at (0,-3) [state] (2) {};
	\node at (2,-3) [state] (3) {};
	\node[above=0pt of 0] {$c_{\{1,2,3\}}$};
	\node[left=0pt of 12] {$c_{\{1,2\}}$};
	\node[right=0pt of 13] {$c_{\{1,3\}}$};
	\node[right=0pt of 23] {$c_{\{2,3\}}$};
	\node[left=0pt of 1] {$c_{\{1\}}$};
	\node[right=0pt of 2] {$c_{\{2\}}$};
	\node[right=0pt of 3] {$c_{\{3\}}$};
	\draw[action, left] (0) to node {$q_3$} (12);
	\draw[action, left] (0) to node {$q_2$} (13);
	\draw[action, right] (0) to node {$q_1$} (23);
	\draw[action, left] (12) to node {$q_2$} (1);
	\draw[action, left, near start] (12) to node {$q_1$} (2);
	\draw[action, right, near start] (13) to node {$\,q_3$} (1);
	\draw[action, left, near start] (13) to node {$q_1$} (3);
	\draw[action, right] (23) to node {$q_2$} (3);
	\draw[action, right, near start] (23) to node {$\,q_3$} (2);
\end{tikzpicture}
\end{center}

Intuitively, the set $B$, when it is the index of $c_B$, denotes the set of indices of those proposition symbols $p_i$ that \emph{have not yet been reached}.
We then define the strategy $\sigma_{a_1}$ simply to select the action $B$ at $q_0$ when the current state of $\mathcal{T}_k$ is $c_B$. (The action $\void$ is selected elsewhere.) 
It is easy to see that $\sigma_{a_1}$ is a strategy that satisfies $\Phi_k$ on every path.

Note that by using $\mathcal{T}_k$, the agent $a_1$ essentially remembers which \emph{subset} of $\{p_1,\dots,p_k\}$ of proposition symbols have already been reached. But $a_1$ does not have to remember in which order these states have been visited; if the order was remembered as well, then the number of states in $\mathcal{T}_k$ would be the number of $k$-permutations plus the initial  state, resulting in $k!+1$ states. 
\end{example}

We prove the following lemma for the model $\mathcal{M}_k$ constructed in Example~\ref{ex: A model for exponential memory}.

\begin{lemma}\label{memory lemma}
$\mathcal{M}_k\not\models^m \xi_k$ when $m<2^k-1$.
\end{lemma}

\begin{proof}
Let $\sigma_{a_1}$ be a strategy for $a_1$ using a transducer $\mathcal{T}$ with less than $2^k-1$ states. We will show that there is a path in $\paths(q_0,\sigma_{a_1})$ on which $p_i$ is not reached for some $i\in[k]$.

We first make the following two observations (i) and (ii):

\smallskip

(i) Suppose $a_1$ chooses some $B\in\mathcal{P}([k])\setminus\{\emptyset\}$ at $q_0$ for which $i\notin B$ for some $p_i$ that \emph{has not yet been reached}. Now the next state may be $q_B$ where it will loop forever. Since $q_B\notin v(p_i)$, the proposition $p_i$ will never be reached.

(ii) Suppose now that $a_1$ chooses some $B$ at $q_0$ for which $i\in B$ for some $p_i$ that \emph{has already been reached}. Now the next state may be $q_i$ and thereafter the game returns to $q_0$. Since $p_i$ is the only proposition symbol that is true at $q_i$, these transitions did not reach any new proposition symbols.

By the points above, we see that in order to reach all $p_i$, the agent $a_1$ has to choose such a set $B$ at $q_0$ which has the indexes of \emph{exactly those} proposition symbols which have not yet been reached. We denote this behavior of $a_1$ by ($\star$).

Since $\mathcal{T}$ has less than $2^k-1$ states, and $|\mathcal{P}([k])\setminus\{\emptyset\}|=2^k-1$, there must be $B'\in\mathcal{P}([k])\setminus\{\emptyset\}$ which $a_1$ never chooses at $q_0$ when following $\sigma_{a_1}$. Supposing that $a_1$ plays according to ($\star$), it may happen that exactly those $p_i$ for which $i\in [k]\setminus B$ are reached (by visiting the corresponding states $q_i$ ($i\in [k]\setminus B$) and returning to $q_0$ after every visit). But, in this situation it is no longer possible for $a_1$ to follow ($\star$) and thus impossible to reach all $p_i$ for which $i\in B$. 
%
\end{proof}

By Example~\ref{ex: A model for exponential memory} and Lemma~\ref{memory lemma} we immediately obtain the following corollary.

\begin{corollary}\label{the: Lower bound for memory cells}
The perfect recall semantics for $\ATL_{\F}^k$  is not equivalent to $m$-bounded memory semantics for any $m<2^k-1$.
\end{corollary}

By this result, agents may need an exponential number of memory cells with respect to the number of relative atoms (in the Boolean combination). Again, this result holds even in the simple case where $\Phi$ is just a conjunction of reachability objectives $\F p$.
Corollary~\ref{the: Lower bound for memory cells} also implies that the result of Theorem~\ref{the: Upper bound for memory cells 2} is optimal. We leave it open whether the result of Theorem~\ref{the: Upper bound for memory cells 1} could be improved.
%


\subsection{Some remarks on the amount of memory needed for a strategy}


There are several ways in which memory resources play a role in strategies.
Besides the read-only memory needed to encode a strategy, for the execution of that strategy one can distinguish the amounts of memory needed:

(i) to store any possible input of the strategy,

(ii) to compute the value of the strategy function on any given input,

(iii) to execute the strategy in any single play.

\noindent
Generally, these can be very different. Usually, the first one is taken as the measure of the memory consumption of a strategy in terms of the required input size (i.e., memoryless, bounded memory, unbounded/perfect recall), while the second is usually disregarded and strategies are assumed to be computed by -- or even hardwired in -- some external devices (``black boxes''). As for the third measure, which involves both the previous two, we are not aware of any explicit consideration of it in the literature. We will make some brief comparing remarks for the case of bounded memory strategies considered here.

From Corollary~\ref{the: Lower bound for memory cells} we see that agents may need a strategy transducer with $2^k-1$ memory cells when there are $k$ reachability objectives. This is because a strategy is a \emph{global} plan of action---or a look-up table---that must take into account all possible plays.
However, by observing the use of truth function in transition games, we see that \emph{in every single play of the game} only $k-1$ memory cells need to be used. That is, the finite state transducer needs to visit only $k-1$ states on every path (c.f. Example~\ref{ex: A model for exponential memory} and the transducer $\mathcal{T}_k$).
Thus, the state space of the transducer has to be exponential with respect to the number of reachability objectives, but only a linearly large section of the transducer is actually used in every single play.
In fact, the latter is to be expected, in the light of the PTIME complexity of model checking of $\ATL^k$, by Theorem~\ref{the: hierarchy of tractable fragments}.   
This observation suggests that the amount of RAM-type of memory needed to use during the play may be a reasonable measure, alternative to the 
number of states in the transducer encoding the agent's strategy in enforcing or refuting a formula $\ATLplus$ (and for other related logics). Thus, one could argue that agents actually only need to use linear amount of memory in $\ATL^k$, supposing they can manage their memory in a more dynamical (``on-the-fly'') way\footnote{This is also justified from the `human perspective', as people can manage to do, say, 10 tasks by remembering what is already done (by remembering at most $9$ pieces of information) without need for exponential memory capacity (which would be $1023$ memory cells by Theorem~\ref{the: Lower bound for memory cells}).}.

\vcut{
\techrep{
For a better bound of the required memory it would be sufficient to modify slightly the transition games and consider a bound $k$ not on the number of all relative atoms in strategic subformulae, but only of the temporal objectives occurring in them.
}
}

\aamasp{\vspace{-0,1cm}}


\section{Conclusion}
\label{sec:concl}


%
%
\aamasp{
A natural extension of the present work would be to develop \GTS for the full $\ATLs$.
}

In conclusion, we note that the game-theoretic semantics for $\ATLplus$ developed here has both conceptual and technical importance, as it explains better how the memory-based strategies in the compositional semantics can be generated and thus also provides better insight on the algorithmic aspect of that semantics. 

We note that a \GTS for \ATLplus, alternative to the one introduced here, could
be obtained via \GTS for coalgebraic fixed point logic
\cite{vene2006, pattinson2009}. 
However, such a semantics (being designed for more powerful logics) would not directly lead to our \GTS that is custom-made for \ATLplus and would thus
not directly enable the complexity analysis that we require.
Also, that alternative approach would not give a semantics where the construction of 
finite paths only suffices. 
%

%
A natural extension of the present work would be to develop \GTS for the full $\ATLs$. Here the
correspondence with B\"uchi games could be exploited in full.

\bibliographystyle{plain}
\bibliography{VG-ATL}

\begin{thebibliography}{10}

\bibitem{AHK02}
R.~Alur, T.~A. Henzinger, and O.~Kupferman.
\newblock Alternating-time temporal logic.
\newblock {\em J. ACM}, 49(5):672--713, 2002.

\bibitem{BrihayeLLM09}
Thomas Brihaye, Arnaud Da~Costa Lopes, Fran\c{c}ois Laroussinie, and Nicolas
  Markey.
\newblock {ATL} with strategy contexts and bounded memory.
\newblock In S.~Art{\"e}mov and A.~Nerode, editors, {\em Proc. of LFCS 2009},
  volume 5407 of {\em LNCS}, pages 92--106. Springer, 2009.

\bibitem{BJ}
Nils Bulling and Wojciech Jamroga.
\newblock Verifying agents with memory is harder than it seemed.
\newblock {\em {AI} Commun.}, 23(4):389--403, 2010.

\bibitem{CHP06}
K.~Chatterjee, T.A. Henzinger, and N.~Piterman.
\newblock Algorithms for {B}uchi games.
\newblock In {\em 3rd Workshop on Games in Design and Verification}, 2006.

\bibitem{pattinson2009}
Corina C{\^{\i}}rstea, Clemens Kupke, and Dirk Pattinson.
\newblock {EXPTIME} tableaux for the coalgebraic
  \emph{{\(\mathrm{\mu}\)}}-calculus.
\newblock In {\em {CSL} 2009}, pages 179--193, 2009.

\bibitem{GKR-AAMAS2016}
Valentin Goranko, Antti Kuusisto, and Raine R{\"{o}}nnholm.
\newblock Game-theoretic semantics for alternating-time temporal logic.
\newblock In {\em Proc. of AAMAS 2016}, pages 671--679.

\bibitem{GKR-AAMAS2017}
Valentin Goranko, Antti Kuusisto, and Raine R{\"{o}}nnholm.
\newblock Game-theoretic semantics for {ATL$^{+}$} with applications to model
  checking.
\newblock In {\em Proc. of AAMAS 2017}, pages 1277--1285, 2017.

\bibitem{GTS-ATL-2018}
Valentin Goranko, Antti Kuusisto, and Raine R\"onnholm.
\newblock Game-theoretic semantics for alternating-time temporal logic.
\newblock {\em ACM Trans. Comput. Log.}, 19(3):17:1--17:38, 2018.

\bibitem{grade}
Erich Gr{\"{a}}del and Igor Walukiewicz.
\newblock Positional determinacy of games with infinitely many priorities.
\newblock {\em Logical Methods in Computer Science}, 2(4), 2006.

\bibitem{HintikkaSandu97}
Jaakko Hintikka and Gabriel Sandu.
\newblock Game-theoretical semantics.
\newblock In J.~van Benthem and A.~ter Meulen, editors, {\em Handbook of Logic
  and Language}, pages 361--410. 1997.

\bibitem{laroussinie}
Fran{\c{c}}ois Laroussinie, Nicolas Markey, and Ghassan Oreiby.
\newblock On the expressiveness and complexity of {ATL}.
\newblock {\em Logical Methods in Computer Science}, 4(2), 2008.

\bibitem{Mazala01}
Ren{\'{e}} Mazala.
\newblock Infinite games.
\newblock In Erich Gr{\"{a}}del, Wolfgang Thomas, and Thomas Wilke, editors,
  {\em Automata, Logics, and Infinite Games: {A} Guide to Current Research},
  volume 2500 of {\em LNCS}, pages 23--42. Springer, 2001.

\bibitem{vene2006}
Yde Venema.
\newblock Automata and fixed point logic: {A} coalgebraic perspective.
\newblock {\em Inf. Comput.}, 204(4):637--678, 2006.

\bibitem{Vester13}
Steen Vester.
\newblock Alternating-time temporal logic with finite-memory strategies.
\newblock In {\em Proc, of GandALF 2013}, volume 119 of {\em EPTCS}, pages
  194--207, 2013.

\end{thebibliography}
\end{document}